\documentclass[10pt]{amsart}

\usepackage{amsmath,amssymb,amsthm,amsfonts,mathrsfs}
\usepackage{graphicx,epsfig,color,float}
\usepackage[utf8]{inputenc} 
\usepackage{hyperref}
\usepackage{caption}
\usepackage{subcaption}

\def\R{\mathbb R}

\def\N{\mathbb N}

\newcommand\pa{{\mathfrak a}}

\newcommand\pb{{\mathfrak b}}

\def\<{\langle}
\def\>{\rangle}
\newcommand{\fol}{{\mathcal{F}}}

\newcommand{\folh}{{\mathcal{F} }_h}

\newcommand{\dist}{\vec{\mathcal{K}}}

\newcommand{\disth}{\vec{\mathcal{K} }_h}

\def\<{\langle}
\def\>{\rangle}
\newcommand{\ssanalytic}{subanalytic }

\newcommand{\DeltaPerp}{\Delta^{\perp}}

\numberwithin{equation}{section}

\newtheorem{theorem}{Theorem}

\newtheorem{lemma}[theorem]{Lemma}
\newtheorem{proposition}[theorem]{Proposition}
\newtheorem{corollary}[theorem]{Corollary}
\newtheorem{definition}[theorem]{Definition\rm}
\newtheorem{remark}[theorem]{Remark}

\numberwithin{theorem}{section}

\title{The analytic minimal rank Sard Conjecture}

\author{A.~Belotto da Silva}
\address{Universit\'{e} Paris Cit\'{e} and Sorbonne Universit\'{e}, UFR de Math\'{e}matiques, Institut de Math\'{e}matiques de Jussieu-Paris Rive Gauche, UMR7586,
F-75013 Paris, France. Institut universitaire de France (IUF).}
\email{andre.belotto@imj-prg.fr}

\author{A.~Parusi\'nski}
\address{Universit\'e C\^ote d'Azur, CNRS, Labo.\ J.-A.\ Dieudonn\'e, UMR CNRS 7351, Parc Valrose, 06108 Nice Cedex 02, France}
\email{adam.parusinski@unice.fr}

\author{L.~Rifford}
\address{Universit\'e C\^ote d'Azur, CNRS, Labo.\ J.-A.\ Dieudonn\'e,  UMR CNRS 7351, Parc Valrose, 06108 Nice Cedex 02, France}
\email{ludovic.rifford@math.cnrs.fr}

\date{}

\makeindex

\begin{document}

\begin{abstract}
We obtain, under an additional assumption on the subanalytic abnormal distribution constructed in \cite{bprfirst}, a proof of the minimal rank Sard conjecture in the analytic category. It establishes that from a given point the set of points accessible through singular horizontal curves of minimal rank, which corresponds to the rank of the distribution, has Lebesgue measure zero. The minimal rank Sard Conjecture is equivalent to the Sard Conjecture for co-rank $1$ distributions.
\end{abstract}

\maketitle

\section{Introduction}\label{sec:Intro}

The topic of this paper is the \emph{minimal rank Sard Conjecture in sub-Riemannian geometry}. This is a follow-up of our previous work \cite{bprfirst}, where we provide a geometrical setting to study the \emph{Sard Conjecture} in arbitrary dimensions. We rely on \cite[Sections 1.1, 1.2 and 1.4]{bprfirst} for a complete presentation of the Conjecture and its importance.

Let $M$ be a smooth connected manifold of dimension $n\geq 3$ equipped with a distribution $\Delta$ of rank $m< n$ which is \emph{bracket generating}. Recall that the Sard Conjecture is only known in very few cases whenever $\dim M >3$. In fact, all previous results focus on Carnot groups \cite{agrachev14,bv20,lmopv16,montgomery02,ov19,riffordbourbaki,lrtPreprint},. Whenever $\mbox{dim}(M)=3$, much more is known by following a geometrical approach inspired by the foundation paper of Zelenko and Zhimtomirskii \cite{zz95}. In fact, in a joint work with Figalli \cite{bfpr18}, we have proved the strong version of the Sard Conjecture in the analytic three dimensional case, improving a previous result of Belotto and Rifford \cite{br18}. The proof is based on a delicate study of the geometrical properties of the, so-called, \emph{characteristic foliation} introduced in \cite{zz95}. It makes use of methods from symplectic geometry, differential topology and singularity theory. Some of the singularity methods -- such as resolution of singularities of metrics and foliations, and regularity of transition maps -- are constrained to small dimensions.

The goal of this paper is to generalize the heart of the arguments of \cite{bfpr18} to arbitrary dimensions, in order to address the \emph{minimal rank Sard Conjecture}. This is the first step in our program to understand the Sard Conjecture in arbitrary dimensions, as explained in \cite[Section 1.4]{bprfirst}. To this end, we follow the framework introduced in \cite{bprfirst}, where we developed the properties of the natural generalization of characteristic foliations, which we called \emph{abnormal distribution} $\vec{\mathcal{K}}$, see \cite[Theorem 1.1]{bprfirst}. We then proceed to replace several of the delicate singularity arguments from \cite{bfpr18} by subanalytic and symplectic arguments; we highlight the new notion of \emph{witness transverse sections} which we expect to be of independent interest to foliation theory, see Theorem \ref{thm:Witness}. This allows us to prove the minimal rank Sard Conjecture under an additional qualitative hypothesis on the abnormal distribution, which we call \emph{splitability}, see Theorem \ref{thm:Sardminrank} and Definition \ref{def:Splittable}. The remaining difficulty is, non-surprisingly, related to singularity theory: one needs to exclude pathological asymptotic behaviors of a foliation of rank at least $2$ over its singular set, see Section \ref{ex:NonSplittable} for an example.

\subsection{Main result}
Let us briefly recall the main objects introduced in \cite{bprfirst} -- we equally rely on that after-mentioned paper for extended discussion on these objects.

Given an analytic totally nonholonomic distribution $\Delta$ on a real-analytic connected manifold $M$, we consider the {\it nonzero annihilator} of $\Delta$ in the cotagent bundle $T^*M$:
\begin{equation}\label{eq:annihilator}
\DeltaPerp : = \Bigl\{ \pa=(x,p) \in T^*M \, \vert \, p \neq 0 \mbox{ and } p\cdot v =0, \, \forall v \in \Delta(x)\Bigr\}.
\end{equation}
By \cite[Theorem 1.1]{bprfirst}, there exists an open and dense subanalytic set $\mathcal{S}_0 \subset \DeltaPerp$ called essential domain. Its complement, the set $\Sigma = \DeltaPerp \setminus \mathcal{S}_0$, is in fact a proper analytic subvariety of $\DeltaPerp$. Moreover, there exists a subanalytic distribution $\vec{\mathcal{K}}$ over $\DeltaPerp$ which is, in fact, a subanalytic \emph{isotropic foliation} over $\mathcal{S}_0$; we denote the restriction of $\vec{\mathcal{K}}$ to $S_0$ by $\vec{\mathcal{K}}_{|S_0}$. We call $\vec{\mathcal{K}}$ the abnormal distribution because a horizontal path $\gamma : [0,1] \to M$ is singular if, and only if, it admits a lift $\xi: [0,1] \to \DeltaPerp$ which is horizontal with respect to $\vec{\mathcal{K}}$, see \cite[Theorem 1.1(iii)]{bprfirst}. 

In our main result, we will need to add an extra hypothesis on the foliation $\vec{\mathcal{K}}_{|S_0}$, which we call \emph{splittability}. We postpone the precise definition to Section \ref{ssec:introsplittable}, and we anticipate in here that any line foliation is splittable. We can now state our main result:

\begin{theorem}[Minimal rank Sard Conjecture for splittable foliatons]\label{thm:Sardminrank}
Assume that both $M$ and $\Delta$ are real-analytic. If the foliation $\vec{\mathcal{K}}_{|S_0}$ is splittable, then the minimal rank Sard conjecture holds true.
\end{theorem}

The proof of Theorem \ref{thm:Sardminrank} will consist in showing by contradiction that if the set of minimal rank singular horizontal paths from a given point reaches a set of positive Lebesgue measure in $M$, then we can, roughly speaking, lift all those horizontal paths into abnormal curves sitting in the leaves of the foliations given by $\vec{\mathcal{K}}$ on its essential domain and from here get a contradiction. This strategy requires to be able to select from a given set of positive measure contained in a transverse local section of the foliation $\vec{\mathcal{K}}_{|\mathcal{S}_0}$ a subset of positive measure whose all points belong to distinct leaves of $\vec{\mathcal{K}}_{|\mathcal{S}_0}$. A foliation subject to such a selection result will be called {\it splittable}; this is the heuristic behind the definition of splittability given in Section \ref{ssec:introsplittable}.

Thanks to \cite[Theorem 1.1(iv)]{bprfirst}, moreover, we know that the rank of the abnormal distribution restricted to its essential domain $\vec{\mathcal{K}}_{|S_0}$ is less than or equal to $m-2$. Hence, the Minimal rank Sard conjecture holds true whenever $\Delta$ has rank $\leq 3$. Furthermore, the equivalence of the minimal rank Sard Conjecture with the Sard Conjecture in the case of corank-$1$ distributions yields the following immediate corollary:

\begin{corollary}\label{Sardcodim1}
Assume that both $M$ and $\Delta$ are analytic. If $\Delta$ has codimension one ($m=n-1$) and the distribution $\vec{\mathcal{K}}_{|\mathcal{S}_0}$ is splittable, then the Sard conjecture holds.
\end{corollary}

In particular, the Sard Conjecture holds true when $n=4$ and $m=3$. This four dimensional result remains true in the $C^{\infty}$ category, as we show in our follow-up work \cite{bprfollowup}, where we extend a few results (as much as we could) from this work to the smooth category.

Finally, our approach to the proof of Theorem \ref{thm:Sardminrank} requires to lift the set of singular horizontal curves in $M$ to a subset of $\DeltaPerp$ of positive transverse volume with respect to $\vec{\mathcal{K}}$. As a consequence, we cannot prove the Sard conjecture for distribution of corank strictly greater than one yet. For an extended discussion, see \cite[Section 1.4 and 1.5]{bprfirst}.

\subsection{Witness transverse sections}
The proof of Theorem \ref{thm:Sardminrank} follows from a combination of the description of abnormal lifts given in \cite[Theorem 1.1]{bprfirst}; differential geometry methods; and a new result on the existence of special transverse sections for singular analytic foliations, which we call \emph{witness transverse sections}, see Theorem \ref{thm:Witness}.

In fact, roughly speaking, we show that if $\mathcal{F}$ is a singular analytic foliation of generic corank $r$ in a real-analytic manifold $N$ equipped with a smooth Riemannian metric $g$, then we can construct locally, for every point $x$ in the singular set $\Sigma$ of $\mathcal{F}$, a special subanalytic set $X\subset V \setminus \Sigma$ where $V$ is an open neighborhood of $x$, called \emph{witness transverse section}. This section has the property that its slices $X^c:=X\cap h^{-1}(c)$ ($c>0$) with respect to some nonnegative analytic function $h$ (verifying $\Sigma \cap V=\{h=0\}$) have dimension $\leq r$ with $r$-dimensional volume uniformly bounded (w.r.t $c$) and such that any point of $V$ can be connected to $X^c$ through a horizontal curve (w.r.t. $\mathcal{F}$) of length no greater than $ \ell$ (w.r.t. $g$). We refer to Section \ref{ssec:witness} for further detail.

\subsection{Splittable foliation}\label{ssec:introsplittable}
Let $N$ be a real-analytic manifold of dimension $n\geq 2$ equipped with a smooth Riemannian metric $h$ (not necessary assumed to be complete) and  $\mathcal{F}$ a (regular) analytic foliation of constant rank $d\in [1,n-1]$. Given $\ell >0$, we say that two points $x$ and $y \in N$ are $(\mathcal{F},\ell)$\emph{-related} if there exists a smooth path $\varphi:[0,1] \rightarrow N$ with length $\in [0,\ell]$ with respect to $h$ which is horizontal with respect to $\mathcal{F}$ and joins $x$ to $y$. Note that the $(\Delta,\ell)$-relation is not an equivalence relation, since it is not transitive. Moreover, given a point $\bar{x}\in N$, we call {\it local transverse section at $\bar{x}$} any set $S\subset N$ containing $\bar{x}$ which is a smooth submanifold diffeomorphic to an open disc of dimension $n-d$ and transverse to the leaves of $\mathcal{F}$. 

\begin{definition}[Splittable foliation]\label{def:Splittable}
We say that the foliation $\mathcal{F}$ is splittable in $(N,h)$ if for every $\bar{x} \in N$, every local transverse section $S$ at $\bar{x}$ and every $\ell>0$, the following property is satisfied: 

\medskip
\noindent
For every Lebesgue measurable set $E\subset S$ with $\mathcal{L}^{n-d}(E)>0$, there is a Lebesgue measurable set $F\subset E$ such that $\mathcal{L}^{n-d}(F) >0$ and for all distinct points $x, \,y \in F$, $x$ and $y$ are not $(\mathcal{F},\ell)$-related.
\end{definition}

We provide in Section \ref{ssec:splittable} a sufficient conditions for a foliation to be splittable. Indeed, we introduce the notion of foliation having locally horizontal balls with finite volume (with respect to the metric $h$ in $N$), see Definition \ref{def:FiniteVolume}, and we prove that this property implies the splittability, see Proposition \ref{PROP:FiniteVolume2}. As a consequence, we infer that every line foliation is splittable, as well as every foliation whose leaves have Ricci curvatures uniformly bounded from below. In particular, all regular foliations in a compact manifold are splittable. An example of non-splittable analytic foliation in a non-compact manifold equipped with a smooth metric is presented in Section \ref{SECEXfoliation}; we do not know if such examples do exist with an analytic metric.

\subsection{Paper structure} The paper is organized as follows: in Section \ref{ssec:splittable}, we discuss the notion of spplitable foliations. In section \ref{ssec:witness} we construct \emph{Witness transverse sections} for analytic foliations, see Theorem \ref{thm:Witness}. Finally, the proof of Theorem \ref{thm:Sardminrank} is given in Section \ref{SECSardminrank}.

\medskip
\noindent
\textbf{Acknowledgment:} The first author is supported by the project ``Plan d’investissements France 2030", IDEX UP ANR-18-IDEX-0001, and partially supported by the Agence Nationale de la Recherche (ANR), project ANR-22-CE40-0014.


\section{Splittable foliations}\label{ssec:splittable}

\subsection{Sufficient conditions for splittability}
The notion of splittable foliation and of $(\mathcal{F},\ell)$-related points have been provided in the Introduction, see Definition \ref{def:Splittable}, and we follow its notation. Let us start our discussion by providing a more structured way to describe two $(\mathcal{F},\ell)$-related points. In fact, let us consider horizontal balls with respect to $\mathcal{F}$ and $h$. Given $x\in N$, we denote by $\mathcal{L}_x$ the leaf of $\mathcal{F}$ through $x$ in $N$. Then, for every $\ell>0$, we call {\it horizontal ball with respect to $\mathcal{F}$ and $h$} the subset of $\mathcal{L}_x$ given by
\[
\mathcal{L}_{x}^{\ell}:= \left\{y \in L_x \, \vert \, \exists \,  \varphi:[0,1] \to \mathcal{L}_x \mbox{ abs. cont. s.t. }  \varphi(0)=x, \, \varphi(1)=y,\, \mbox{length}^h(\varphi) \leq \ell \right\}.
\]
We check easily that two points $x,y \in N$ are  $(\mathcal{F},\ell)$-related if and only if $y\in \mathcal{L}_x^{\ell}$ (or $x\in \mathcal{L}_y^{\ell}$). Let us now introduce the following definition where $\mbox{vol}^{h,\mathcal{F}}(A)$ stands for the volume of a Borel set $A$ contained in a leaf of $\mathcal{F}$ with respect to the Riemannian metric induced by $h$ on that leaf:

\begin{definition}\label{def:FiniteVolume}
We say that the foliation $\mathcal{F}$ has locally horizontal balls with finite volume (w.r.t. $h$) if for every $x\in N$ and every $\ell>0$, there are $V>0$ and a neighborhood $\mathcal{U}$ of $x$ such that $\mbox{\em vol}^{h,\mathcal{F}}(\mathcal{L}_{y}^{\ell}) \leq V$ for all $y\in \mathcal{U}$.
\end{definition}

The first example of foliations having locally horizontal balls with finite volume is given by foliations associated with complete Riemannian metrics. As a matter of fact, if $h$ is complete, then by the Hopf-Rinow Theorem, all balls $\mathcal{L}_{y}^{\ell}$ with $y$ close to $x$ are contained in the ball centered at $x$ with radius $\ell+1$ which happens to be compact, so all of those horizontal balls are compact sets with a volume which is finite and depends continuously upon $y$.  Another example is given by foliations whose curvature satisfy a lower bound:

\begin{proposition}\label{PROP:FiniteVolume}
If $\mathcal{F}$ has rank $1$ then it has locally horizontal balls with finite volume, indeed we have for  any $x\in N$ and $\ell>0$, $\mbox{\em vol}^{h,\mathcal{F}}(\mathcal{L}_{x}^{\ell}) \leq 2\ell$. Moreover, if $\mathcal{F}$ has rank $\geq 2$ and the Ricci curvature (w.r.t. $h$) of all its leaves is uniformly bounded from below, then it has locally horizontal balls with finite volume (w.r.t. $h$).
\end{proposition}

The proof of this result is left to the reader. We draw their attention to the fact that the comparison theorem required for the proof (of the second part) remains true for a non-complete metric (see {\it e.g.} \cite[\S 4]{chavel06}). We end this section with the result that justifies the introduction of Definition \ref{def:FiniteVolume} and provide many examples of splittable foliations.

\begin{proposition}\label{PROP:FiniteVolume2}
If $\mathcal{F}$ has locally horizontal balls with finite volume (w.r.t. $h$), then it is splittable in (N,h).
\end{proposition}
\begin{proof}
Let $\bar{x}\in N$ and $\ell_0>0$ be fixed. Let $\ell>\ell_0$ and denote by $\mathcal{L}_{x}^{<\ell} \subset \mathcal{L}_{x}^{\ell} $ the union of all $\mathcal{L}_{x}^{\ell'}$ with $\ell' < \ell$. Let $V>0$ be such that $\mbox{vol}^{h,\mathcal{F}}(\mathcal{L}_{x}^{4\ell}) \leq V$ for all $x$ in an open neighborhood $U$ of $\bar{x}$. By considering a foliation chart (see \cite[Section 3]{bprfirst}) and shrinking $U$ if necessary, there exists a diffeomorphism $\Phi: W \to U$ such that $W = (-1,1)^{n} \subset \mathbb{R}^n$, $\Phi(0)= \bar{x}$, the pull-back foliation is given by $(x_{n-d+1} = \ldots = x_n=cte)$ (where $d$ is the rank of $\mathcal{F}$) and the following properties are verified:
\begin{itemize}
\item there exists a smooth transverse section $\mathcal{D}$ diffeomorphic to a disc of dimension $n-d$; 
\item there exists $\epsilon>0$ such that, for every point $x\in \mathcal{D}$, the connected component of $\mathcal{L}_x \cap U$ containing $x$, which we denote by $\mathcal{L}_{x,U}$, is such that $\mbox{vol}^{h,\mathcal{F}}(\mathcal{L}_{x,U}) >\epsilon$ and $\mbox{diam}(\mathcal{L}_{x,U})<\ell$ (where $\mbox{diam}$ stands for the diameter w.r.t. $h$).
\end{itemize}
Let $K$ be a natural number greater than $V/\epsilon$. By construction, if $x,y \in \mathcal{D}$ are $(\mathcal{F},\ell)$-related then we have $\mathcal{L}_{x,U}, \mathcal{L}_{y,U} \subset \mathcal{L}_x^{4\ell}$. Thus, since $\mbox{vol}^{h,\mathcal{F}}(\mathcal{L}_{x,U}) >\epsilon$ for every $x\in \mathcal{D}$ and $\mbox{vol}^{h,\mathcal{F}}(\mathcal{L}_{x}^{4\ell}) \leq V$ for all $x\in U$, we infer that for every $x\in \mathcal{D}$, there are at most $K$ points in $\mathcal{D}$ which are $(\mathcal{F},<\ell)$-related to $x$ (that is, they are $(\mathcal{F},\ell')$-related to $x$ for some $\ell'<\ell$). Denote by $\{x_1,\ldots,x_{k_x}\}$ these points, where $k_x \leq K$ depends on $x\in \mathcal{D}$. Let $E \subset \mathcal{D}$ be a measurable set such that $\mathcal{L}^{n-d}(E)>0$. Let $k$ be the maximum value of $k_x$ for every $x\in E$ which is a density point of $E$. Fix a density point $x\in E$ such that $k_x = k$ and consider the set $\{x_1,\ldots,x_{k}\}$ of $(\mathcal{F},<\ell)$-related points to $x$ in $\mathcal{D}$. Denote by $\varphi_i:[0,1] \to \mathcal{L}_x$, for $i=1,\ldots,k$, the absolutely continuous curves of length $<\ell$ between $x$ and $x_i$ respectively. Since $\mathcal{F}$ is everywhere regular and $\varphi_i$ has compact domain, we conclude from the foliation charts that there exists a transverse section $\mathcal{D}_x \subset \mathcal{D}$ containing $x$ and diffeomorphic to a disc of dimension $n-d$, such that: for every $y\in \mathcal{D}_x$ the curves $\varphi_i$ can be diffeomorphically deformed into an absolutely continuous curve $\widetilde{\varphi_i}:[0,1] \to \mathcal{L}_y$ starting from $y$ and finishing at a point $y_i \in \mathcal{D}$ with length $<\ell$, for every $i=1,\ldots,k$. Now, since all the points $\{x,x_1,\ldots,x_k\}$ are distinct, apart from shrinking $\mathcal{D}_x$, we may suppose that for every $y \in \mathcal{D}_x$, all other points $\{y_1,\ldots,y_k\}$ do not belong to $\mathcal{D}_x$. We now consider $F = E\cap \mathcal{D}_x$. First, note that $\mathcal{L}^{n-d}(F)> 0$ since $x$ is a density point. Moreover, for $y \in F$, we know that $k_y \geq k$ since $y \in \mathcal{D}_x$, and that $k_y \leq k$ since $y \in E$. We conclude easily that every two points of $F$ are not $(\mathcal{F},\ell_0)$-related, finishing the proof.
\end{proof}

As a consequence of Propositions \ref{PROP:FiniteVolume} and \ref{PROP:FiniteVolume2}, we get the following result:

\begin{proposition}\label{PROPrankonesplit}
Every foliation of rank $1$ is splittable. 
\end{proposition}

We provide in Section \ref{SECEXfoliation} an example of analytic foliation $\fol$ contained in an analytic manifold with boundary $M$, endowed with a (non-complete) $C^{\infty}$ metric $g$, which is non-splittable. This example illustrates the kind of qualitative behavior that we must exclude when studying the minimal rank Sard Conjecture. Nevertheless, note that the example is constructed on an abstract manifold. We do not know the answer to the following question:

\medskip
\noindent {\bf Open question.} Is there an integrable family of analytic $1$-forms $\Omega = (\omega_1,\ldots,\omega_t)$ defined over an open set $U \subset \mathbb{R}^n$ whose associated analytic foliation $\fol$ defined in $U \setminus \Sigma $, where $\Sigma$ is the singular set of $\Omega$, is non-splittable in $(U,g^0)$ where $g^0$ is the Euclidean metric?

\medskip
If the answer to the above question is negative, then the hypothesis of Theorem \ref{thm:Sardminrank} would always be satisfied provided that $M$ and $\Delta$ are analytic.

\subsection{Example of a non-splittable foliation}\label{SECEXfoliation}

\label{ex:NonSplittable}
We modify a construction of Hirsch \cite{Hirsch74} in order to define a foliation which is non-splittable in a (non-compact) manifold with border $M$. As a matter of fact, Hirsch foliations are two-dimensional analytic foliations which satisfy the topological properties of a non-splittable foliation, but they lack the metric properties. In order to obtain the metric properties, we modify the original construction, and we make use of $C^{\infty}$-partitions of the unit to yield a $C^{\infty}$-metric.

We start by defining the building-blocks. Consider the double cover immersion $f:S^1 \to S^1$ given by $f(t) = 2t$, and choose an analytic embedding $\iota$ of the solid torus $S^1 \times D^2$ onto its interior so that $\pi \circ \iota = f \circ \pi$, where $\pi: S^1 \times D^2 \to S^1$ is the projection. Let $V = S^1 \times D^2 \setminus \mbox{Int}(\iota(S^1 \times D^2))$. Then the boundary of $V$ is two copies of $S^1 \times S^1$, which we denote by $V^{-}$ and $V^{+}$ where $\iota(V^{-}) = V^{+}$. Denote by $\mathcal{G}$ foliation over $V$ induced by the the fibration $\pi$. Note that the leaves of this foliation are topological pants, whose intersection with $V^{-}$ is a $S^1$, and whose intersection with $V^{+}$ is the disjoint union of two $S^1$, cf. figure \ref{sfig:ExHirschBuildingBlock}.

\begin{figure}[H]
\centering
\begin{subfigure}[b]{0.45\textwidth}
\centering
\includegraphics[scale=0.3]{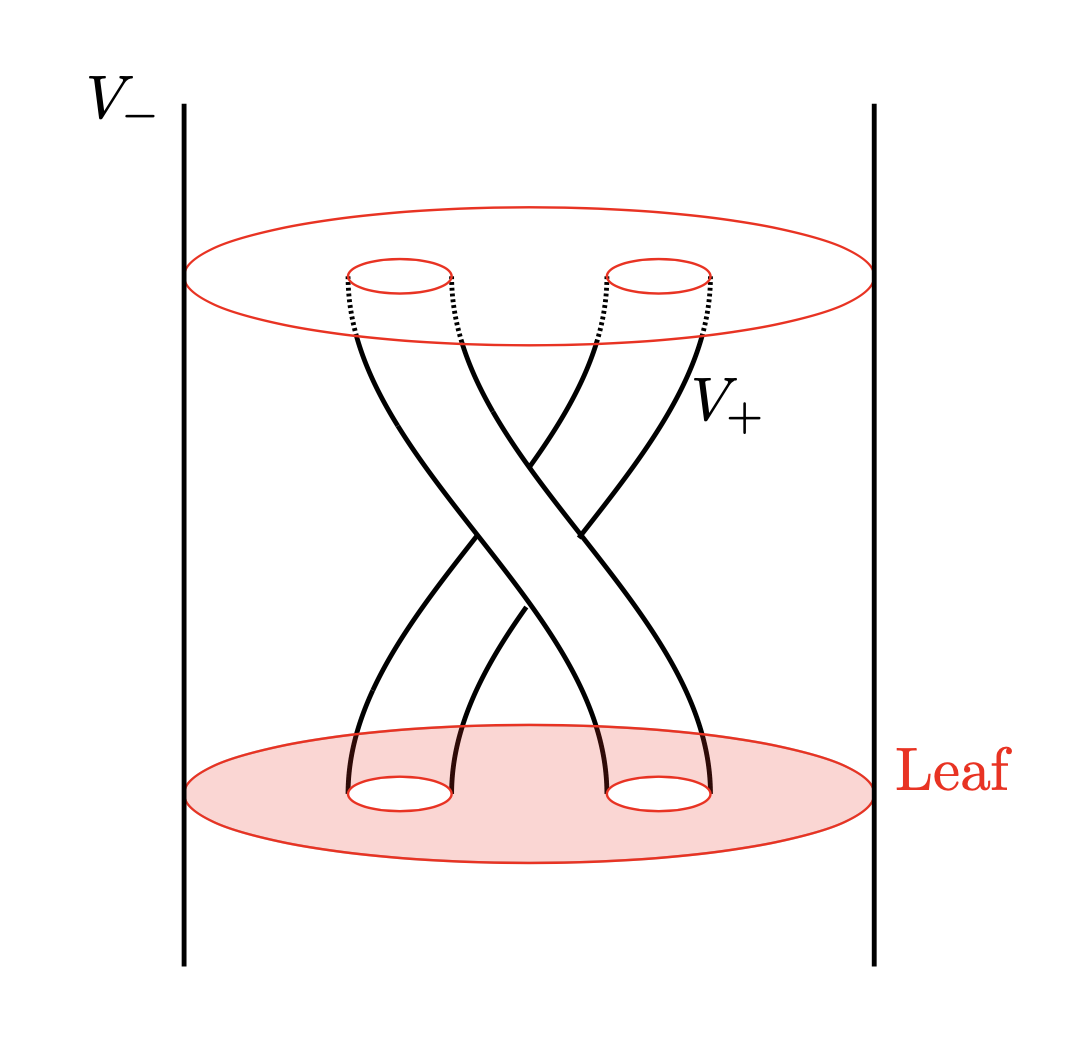}
\caption{The building block $V$.}
\label{sfig:ExHirschBuildingBlock}
\end{subfigure}
\begin{subfigure}[b]{0.45\textwidth}
\centering
\includegraphics[scale=0.3]{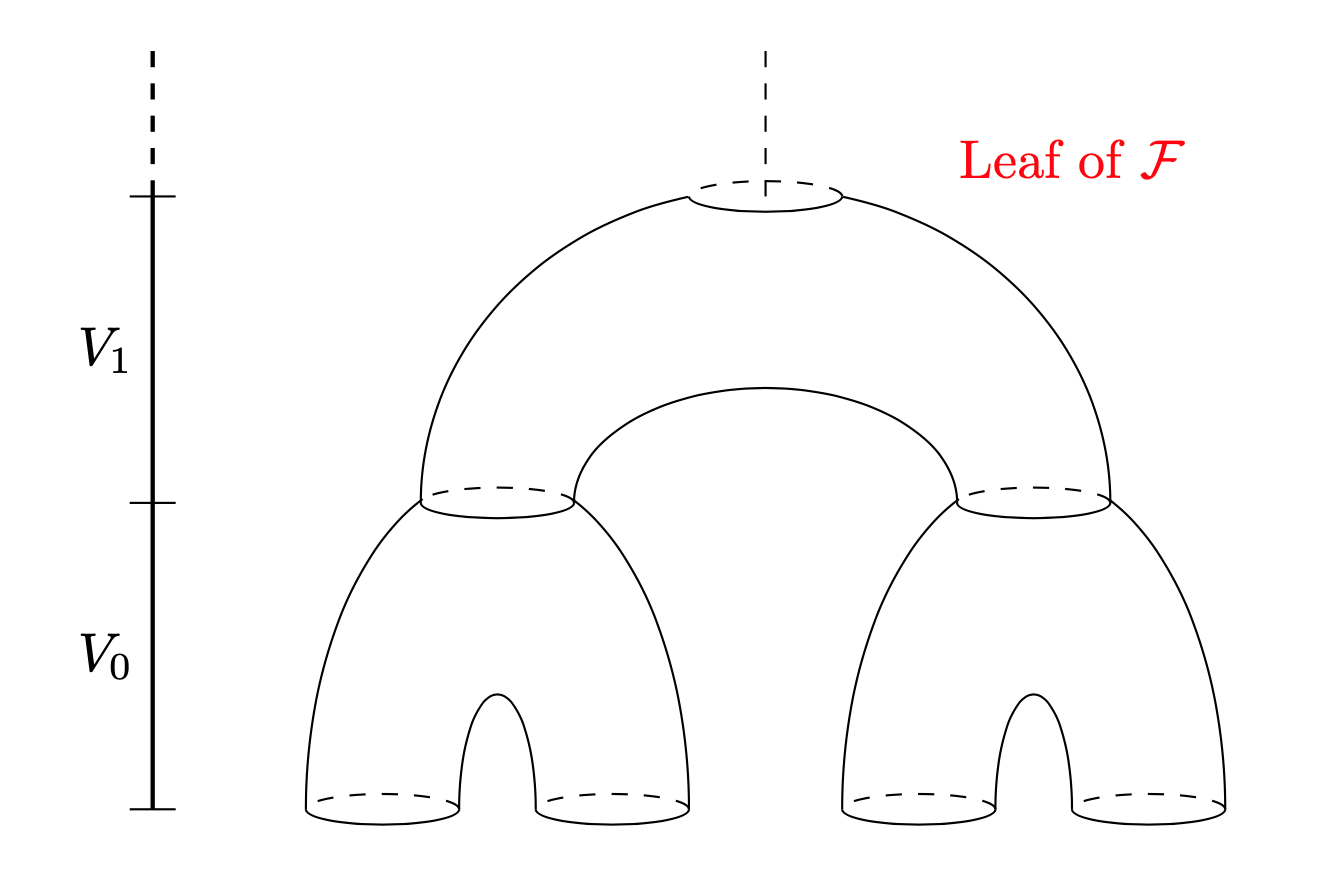}
\caption{A leaf of $\mathcal{F}$.}
\label{sfig:ExHirschLeaf}
\end{subfigure}
\caption{Geometrical illustration of $\mathcal{F}$.}
\label{fig:ExHirsch}
\end{figure}

Now, we consider a countable family of building-blocks $(V_n, V_n^{-}, V_n^{+},\iota_n, \mathcal{G}_n,g_n)$, where $g_n$ are analytic metrics over $V_n$ satisfying the following property: given two points $x$ and $y$ in a leaf $L$ of $\mathcal{G}_n$, the distance of $x$ and $y$ in $L$ is bounded by $4^{-n}$. We denote by $M$ the manifold with boundary given by the union of all $V_n$, by identifying $V_n^{-}$ with $V_{n+1}^{+}$ via $\iota$, that is, we take the identification $x\in V_n^{-}$ equivalent to $\iota(x) \in V_{n+1}^{+}$. This yields an analytic manifold with border, where the border is a torus $M_0 = V_0^{+} = S^1 \times S^1$. This construction induces, furthermore, an analytic foliation $\mathcal{F}$ over the manifold with border $M$ which locally agrees with $\mathcal{G}_n$ over each $V_n$, because $\pi \circ f = \iota \circ\pi$, cf. figure \ref{sfig:ExHirschLeaf}. Furthermore, we can define a globally defined $C^{\infty}$ metric $g$ over $M$ by patching the metrics $g_n$ via partition of the unit. We may chose such a partition so that $g$ satisfies the following property: given two points $x$ and $y$ in a leaf $L$ of $\mathcal{G}_n$, the distance of $x$ and $y$ in $L$ is bounded by $2^{-n}$.

We claim that $\mathcal{F}$ is a non-splittable foliation. Indeed, consider a transverse section $\Sigma = S^1 \subset M_0 = S^1 \times S^1$ and let us identify $\Sigma$ with the interval $[0,1]$. Given a point $x \in \Sigma$, denote by $L_x$ the leaf passing by $x$. First, consider the foliation $\mathcal{G}_0$, and note that, since $f(x) = f(x+1/2)$ and $\mathcal{G}_0$ is a foliation by pants, $x+ 1/2$ also belongs to the leaf $L$, cf. figure \ref{sfig:ExHirschLeaf}. Since this argument can be iterated over any $\mathcal{G}_n$, we get that all points $x + m/2^n$ with $m,n \in \mathbb{N}$ belong to $L_x$. Moreover, the distance on $L_x$ between $x$ and $x+ m/2^n$ is bounded by:
\[
2 \cdot  \sum_{k=0}^n \frac{1}{2^k} < 4
\]
since there exists a path between $x$ and $x+ m/2^n$, contained in the leaf $L_x$, and which is contained in the union of $V_k$ with $k< n$, crossing each of these components at most twice. In conclusion, for every $x\in \Sigma$, the intersection of $L_x$, the leaf passing through $x$, with $\Sigma$ is a countable and dense set of points invariant by a countable subgroup of rotations, which are pairwise $(\mathcal{F},4)$-related. We infer that $\mathcal{F}$ is a not splittable in $(M,g)$ because there is no measurable set $E\subset \Sigma$ with positive Lebesgue measure whose intersection with each $L_x$ (with $x\in \Sigma$) contains only one point.

\section{Witness transverse sections}\label{ssec:witness}

We follow the notation and we use results given in \cite[Section 3.3 and 3.4]{bprfirst}. The main goal of this section is to show the following result:

\begin{theorem}\label{thm:CorWitness}
Let $N$ be a real-analytic manifold of dimension $d\geq 1$ equipped with a complete smooth Riemannian metric $g$, $\Omega$ be a family (or, more generally, a sheaf) of analytic $1$-forms which is integrable of generic corank $r$, with singular set $\Sigma$.  
Denote by $\dist$ the distribution associated to $\Omega$ and by $\fol$ the foliation  on $N\setminus \Sigma$ associated to $\Omega$. Let $x \in N$ and $\ell>0$ be fixed. Then, there exist a relatively compact open neighborhood $V$ of $x$ in $N$, a real-analytic function $h: V \to [0,\infty)$, a \ssanalytic  set $X\subset V\setminus \Sigma$ and $C, \epsilon >0$ such that the following properties are satisfied:
\begin{itemize}
\item[(i)] 
The set $h^{-1}(0)$ is equal to $\Sigma \cap V$.
\item[(ii)] 
(uniform volume bound)
$\dim X\le r+1$ and for every $0<c<\epsilon$ the \ssanalytic set $X^c:=X\cap h^{-1}(c)$ satisfies $\dim X^c \le r$ and its $r$-dimensional volume with respect to $g$ is 
bounded by $C$.

\item[(iii)] 
(uniform intrinsic distance bound)
For every $0<c<\epsilon$ and for every $\pa \in h^{-1}(c) \subset  V \setminus \Sigma$, there is a smooth curve $\alpha :[0,1] \rightarrow V\setminus \Sigma$  contained in $\mathcal{L}_{\pa} \cap h^{-1}(c)$, where $\mathcal{L}_{\pa}$ is the leaf of $\fol$ containing $\pa$, such that
\[
\alpha(0)=\pa, \quad \alpha(1)\in X^c, \quad \mbox{and} \quad \mbox{\em length}^{{g}}(\alpha) 
\leq \ell.
\]
\item[(iv)] 
(generic tranversality) 
There is a decomposition $X=Y\sqcup Z$ as the disjoint union of two \ssanalytic sets $Y,Z$ such that: $\dim Z\le r$, 
$Y= \bigsqcup _{i\in I} Y_i$ is a finite disjoint union of smooth \ssanalytic sets $Y_i$ of dimension $r+1$, and, moreover, for every $0<c<\epsilon$, $Z^c:=Z\cap h^{-1}(c)$ is of dimension $<r$, 
$Y_i^c:=Y_i\cap h^{-1}(c)$ is smooth of dimension $r$ such that   
\[
\partial Y^c = \overline{Y^c} \setminus Y^c \subset Z^c \quad \text{ and } \quad  T_{\pa} N = T_{\pa}Y_i^c + \dist (\pa) \qquad \forall \pa \in Y_i^c, \, \forall i \in I.
\]
\end{itemize}
\end{theorem}

First we show the following general theorem on the existence of a transverse section, that is of independent interest for foliation theory. Then Theorem \ref{thm:CorWitness} will be a corollary of this result.

\begin{theorem}\label{thm:Witness}
Let $N$ be a real-analytic manifold of dimension $d\geq 1$ equipped with a complete smooth Riemannian metric $g$, $\Omega$ be a family (or, more generally, a sheaf) of analytic $1$-forms which is integrable of generic corank $r$, with singular set $\Sigma$. Denote by $\dist$ the distribution associated to $\Omega$, and by $\fol$ the foliation on $N\setminus \Sigma$ associated to $\Omega$. Then, for every $x\in N$ there exist a relatively compact open \ssanalytic neighborhood $V$ of 
$x$ in $N$, a \ssanalytic set $X\subset V\setminus \Sigma$, called witness transverse section, such that the following properties are satisfied:
\begin{itemize}
\item[(i)] 
For every $z\in V \setminus \Sigma$ there is a smooth curve $\alpha :[0,1] \rightarrow  V\setminus \Sigma$ contained in a leaf of $\fol$ such that 
\[
\alpha(0)\in X, \quad \alpha(1)=z  \quad \mbox{and} \quad \mbox{\em length}^g(\alpha) \leq C_d  \, 
\mbox{\em diam}^g(V), 
\]
where $C_d$ is a constant depending only on $d$. 
\item[(ii)] 
$X$ is the disjoint union of finitely many locally closed smooth \ssanalytic 
sets $X= \bigcup_i X_i$ of dimension at most $r$ such that 
for every $\pa \in X_i$ we have $\dist(\pa) \cap T_{\pa}X_i = \{0\}$. In particular, if  $Y$ is the union of $X_i$ of maximal dimension $r$ 
and $Z$ the union of those of dimension $< r$ then $X=Y\sqcup Z$ and $Y$ is
 transverse to the leaves of $\fol$.  
 \end{itemize}
\end{theorem}

We may assume, without loss of generality, that the metric $g$ is real-analytic (or even Euclidean with respect to a fixed local coordinate system). Indeed, it is enough to show  the statement of Theorem \ref{thm:Witness} locally at $x$, and any $C^{\infty}$ metric $g$ is locally bi-Lipschitz equivalent to the Euclidean metric and, moreover, the bi-Lipschitz constant may be taken arbitrarily close to $1$.

Given a small open neighborhood $V$ of $x\in N$ there is a finite family of analytic functions on $V$, $\mathcal G=\{g_i\}$, such that: 
\begin{enumerate}
\item [(i)] each $g_i$ is Lipschitz with constant $2$ (with respect to the geodesic distance $d^g$),  
\item [(ii)]
for every  $x\in V$, for every smooth submanifold $M\subset V$ and  every vector 
$v\in T_N$ there is an index $i$ such that
\begin{enumerate}
	\item [(1)]
	$|\nabla (g_{i|M})(x)|\ge 1/2$,
	\item [(2)]
	$\langle \nabla (g_{i|M})(x), v\rangle \ge 0$.
	\end{enumerate}
\end{enumerate}
Indeed, first note that it suffices to show that existence of the family $\{g_i\}$ satisfying 
$(i)$ and $(1)$ of $(ii)$.  The conditions $(2)$ of $(ii)$ can be obtained by adding to the family $\{g_i\}$ all the opposite functions $\{- g_i\}$.  
To get $(i)$ and $(1)$ of $(ii)$, by the preceding remark, it suffices to consider only the case when $N= \mathbb{R}^n$ and $g$ is the Euclidean metric. In this case, let us identify $\mathbb{S}^{n-1}$ with the vectors $v \in\mathbb{R}^n$ of norm one. For each fixed $w\in \mathbb{S}^{n-1}$ there exists an open neighborhood $U_w$ of $w$ in $\mathbb{S}^{n-1}$ such that, for all $v\in U_w$, $\langle w, v\rangle > 1/2$. By the compactness of $\mathbb{S}^{n-1}$, consider a finite family $w_i$ such that $U_{w_i}$ covers $\mathbb{S}^{n-1}$, and take as $\mathcal{G}$ the family of functions $g_{i}(x) =  \langle w, x\rangle$. Since $|\nabla g_{i}(x)| =|w_i|=1$, the family clearly satisfies $(i)$. Next, fixed a submanifold $M$ and a vector $v \in T_xM$ of norm $1$, there exists $i_0$ such that $|\nabla (g_{i_0|M})(x)| \ge |\langle \nabla g_{i_0}(x),v \rangle| =|\langle w_{i_0},v \rangle | \ge 1/2$ by construction, proving $(ii)$.

\begin{remark}
There is a family $\mathcal G=\{g_i\}$ of functions defined on the entire $N$ that satisfies the above properties $(i)$ and $(ii)$ at every point $x\in N$. Indeed, one may show it first in the class of $C^1$ functions and then approximate them by real analytic ones in Whitney $C^1$-topology, see e.g. \cite{gg73}. Therefore, in this case,  a stronger version of Theorem \ref{thm:Witness} holds, where we may take $V = N$ and replace $\mbox{\em diam}^g(V)$ by an arbitrary constant $D>0$. We do not need this stronger result in this paper. 
\end{remark}

We now consider an auxiliary locally closed nonsingular \ssanalytic subset $W$ of $N\setminus \Sigma$. Note that we do not fix its dimension; in fact, we will make inductive arguments based on its dimension. Later on, in the proof of Theorem \ref{thm:Witness}, $W$ will be taken as a connected component of $V\setminus \Sigma$. Following the notion introduced in \cite[Section 3]{bprfirst} we say that is $\dist$ \emph{is regular on ${W}$} if the restriction of $\dist$ to ${W}$ is a regular analytic distribution and $\dist$ has constant rank along ${W}$. We denote by $\dist_{W}$ this restriction and by $r_{W}$ its corank.  By \cite[Remark 3.11]{bprfirst}, $r_{W}\le r$,  $\dist_{W}$ is integrable and induces a foliation that we denote by $\fol_{W}$. We start by studying what happens when $\dist$ is regular on $W$:

\begin{lemma}\label{lem:existenceoftheset1}
Following the above notation, suppose that $\dist$ is regular on $W$ and $r_{W} < \dim {W}$.
Then there exists a \ssanalytic (as a subset of $N$) subset $Y_{W}\subset {W}$ of dimension $< \dim {W} $ such that 
for every $z\in {W}$ there is a smooth curve $\alpha :[0,1] \rightarrow {W}$, contained entirely in a leaf of $\fol_{W}$ such that 
\[
\alpha(0)\in Y_{W}, \quad \alpha(1)=z  \quad \mbox{and} \quad \mbox{\em length}^g(\alpha) \leq 4 \, 
\mbox{\em dist}^ g (z, \alpha(0)).
\] 
\end{lemma} 

\begin{proof}
We work locally in a relatively compact neighborhood $V$ of $x\in \overline {W}$. 
Let $f$ be a $C^2$ subanalytic function such that $f^{-1} (0) = (\overline {W} \setminus {W}) \cup 
(\overline V \cap {W} \setminus V) $.  
Such a function, even a function of class $C^p$ for any fixed finite $p$, always exists.  
It follows from a more general result valid in any o-minimal structure, see Theorem C11 of 
\cite{dm96}.  The subanalytic case, that we use here, was proven first by Bierstone, Milman and Paw{\l}ucki (unpublished).  By replacing $f$ by $f^2$ we may suppose $f\ge 0$. Fix a family $\mathcal{G}= \{g_i\}$ of analytic functions as above and define 
\begin{multline*}
 Y_i:=  \mbox{Bd}_{W} \left( \left\{z\in {W} \, ; \, |\nabla (g_{i|\fol_{W}})(z)|=  1/2 \right\} \right) \\ 
 \cup 
 \mbox{ Bd}_{W} \left(  \left\{z\in {W} \, ; \, 
 |\langle \nabla (f_{|\fol_{W}})(x), \nabla (g_{i|\fol_{W}})(x)\rangle = 0\right\} \right),
\end{multline*}
 where by $\mbox{Bd}_{W}$ we mean the topological boundary in ${W}$.  
Here by $\nabla (f_{|\fol_{W}})(z)$ we mean the gradient of the function: $f$ restricted to the leaf of $\fol_{W}$ through $z$. These leaves are of dimension $\ge 1$ by the assumption $r_{W} < \dim {W}$.  

The sets $Y_i$ are \ssanalytic (the Riemannian metric $g$ is assumed real-analytic) and of dimension $ < \dim {W}$. 
Then we take as $Y_W$ the union of all $ Y_i$.

Let $z\in ({W}\setminus Y_{W})\cap V$ be fixed. By the above property (ii), there is $i$ such that 
\[
|\nabla (g_{i|F})(z)|\ge 1/2 \quad \mbox{and} \quad \langle \nabla (f_{|F})(z), \nabla (g_{i|F})(z)\rangle \ge 0,
\]
where $F$ is the leaf of $\fol_{W}$ containing $z$. Let $\beta:[0,t_0)\to {W}\setminus Y_{W}$ be the maximal integral curve of $\nabla (g_{i|F})$ with $\beta(0)=z$ (the curve $\alpha$ from the Lemma will be later taken as a reparametrization of $\beta$).  It is of finite length.  Indeed, for any $t_1\in [0,t_0)$, we have (note that by construction of $Y_{W}$, $|\nabla (g_{i|F})(\beta (t)|\geq 1/2$ for all $t\in [0,t_0)$)
\begin{multline}\label{eq:boundbylength}
g_i(\beta (t_1)) - g_i (z) =  \int_0^{t_1} \left|\nabla (g_{i|F})(\beta (t)\right|^ 2 \, dt \\
\ge 
\frac{1}{2}  \int_0^{t_1} \left|\nabla (g_{i|F})(\beta (t)\right| \, dt = \frac{1}{2} \, \mbox{length}^g (\beta([0,t_1]).
\end{multline}
Therefore, since $g_{i|V}$ is bounded, $t_0$ is finite and  $\lim_{t\to t_0} \beta (t)$ exists.  We denote it by $\beta(t_0)$.  Because $f(\beta (t))$ is not decreasing it is not possible that $\beta(t_0)\in (\overline {W} \setminus {W}) \cup 
(\overline V \cap {W} \setminus V)$ and therefore  $\beta(t_0) \in Y_{W}$.  
 Moreover, since $g_i$ is $2$-Lipschitz, \eqref{eq:boundbylength} yields
$$
\mbox{length}^g (\beta) \le 2 \left(g_i(\beta (t_0) ) - g_i (z)\right)  \le 4 \,d^ g (\beta (t_0), z ) . 
$$
The curve $\beta$ is analytic except, maybe, at $t_0$. In this case we reparametrize it to obtain a smooth curve $\alpha$ satisfying the statement. 
\end{proof}

\begin{remark}
Lemma \ref{lem:existenceoftheset1} implies that every leaf of $\fol$ intersecting ${W}$ intersects $Y_{W}$.  
A similar result was shown in the definable set-up in \cite{speiss99} under an additional assumption that the leaves of $\fol$ are Rolle, see \cite[Proposition 2.2]{speiss99}.  
This extra assumption implies that the leaves are locally closed that is not the case in general.
\end{remark}

More generally, recall that we say that a stratification $\mathcal{S} = (S_{\alpha})$ of ${W}$ is compatible with the distribution $\dist$ if for every stratum $S$ of $\mathcal{S}$, $\dist $ is regular on $S$.  In this case, for every stratum $S$, denote the restriction of $\dist$ to $S$ by $\dist_S$, and its corank, which is constant on $S$, by $r_S$.

\begin{proposition}\label{prop:existenceoftheset}
Let ${W}$ be a locally closed relatively compact nonsingular \ssanalytic subset of $N\setminus \Sigma$. There exists a \ssanalytic
stratification $\mathcal{S}$ of ${W}$, compatible with $\dist$,  satisfying the following property: let $X_0$ denote the union of all strata $S'$ of $\mathcal{S}$ for which $\dim S'=r_{S'}$ (i.e. the leaves of $\fol_{S'}$ are points), then for every $z\in {W}$ there is a smooth curve $\alpha :[0,1] \rightarrow  {W}\setminus \Sigma$, contained in the leaf of $\mathcal{F}$ through $z$ such that 
\[
\alpha(0)\in X_{0}, \quad \alpha(1)=z  \quad \mbox{and} \quad \mbox{\em length}^g(\alpha) \leq C_d \, 
\, \mbox{\em diam}^g ({W}).
\]
\end{proposition}

Note that a refinement of a stratification satisfying the conclusion of the above proposition also satisfies this conclusion.
Therefore, if we want this stratification to satisfy additional properties, to be Whitney for instance, we replace it by its  refinement.  

\begin{proof}
Induction on $\dim W$.
Let $\mathcal{S}$ be a \ssanalytic stratification of $W$ compatible with $\dist$. The cases of $\dim {W}=0$ or $r_S = \dim S$ for all strata $S$ are obvious. Therefore we may assume that there is a stratum $S$ such that $r_S < \dim S$ and then we apply Lemma \ref{lem:existenceoftheset1} to $S$. Let $Y_S$ be the set given by this lemma, which has dimension smaller than $\dim W$.  We stratify $Y_S$ (it is not necessarily nonsingular) and apply the inductive assumption to every stratum. We repeat this procedure for all strata $S$ such that $r_{S}< \dim S$. The obtained stratification satisfies the statement for $S$.

Indeed, let us prove the  last property, namely the bound on the length of $\alpha$. Let $z\in S$.  By Lemma \ref{lem:existenceoftheset1} we may connect $z$ and a point of 
$y\in Y_S$ by an arc in a leaf of $\fol$ of length $\le 4 \mbox{diam}^g (S)$.  The point $y$ belongs to a stratum of smaller dimension and we may use to it the inductive assumption. So finally we may connect $z$ to a point of $X_{0}$ by an arc of $\le 4 ^{\dim S} \mbox{ diam}^g (S)$. Since every leaf of $\mathcal{F}$ is smooth, and this arc has at most $d$ non-smooth points, we can reparameterize it by an everywhere smooth arc without increasing its length.  
It shows that we may choose $C_d = 4^d$. 
\end{proof}


\subsubsection{Proof of Theorem \ref{thm:Witness}.}

Let $V$ be a \ssanalytic open relatively compact connected subset of $N$.   
 Let $X_{0}\subset \overline V \setminus \Sigma$ be the set given by Proposition \ref{prop:existenceoftheset} for ${W}=V\setminus \Sigma$. Then, $X=X_0$ satisfies (i) of theorem. The condition (ii) of the theorem follows directly from the property that the leaves of $\fol_{S'}$ are points, for all strata $S'$ contained in $X_0$. 
\qed

\subsubsection{Proof of Theorem \ref{thm:CorWitness}.}
 
Let $h$ be an analytic function defined in a neighborhood of $x$ such that $h^{-1} (0) = \Sigma$. Denote the distribution defined by $dh$ and 
$\omega_i, i\in I$, by ${\disth} $.  Its singular locus equals 
$\Sigma_1=\Sigma \cup \Sigma_h$, 
where 
$$\Sigma_h= \left\{x\in V\setminus \Sigma; dh (x) \in \disth ^{\perp}\right\}, $$
and  ${\disth} $ is integrable of corank $r+1$ in its complement. We denote the induced foliation by $\folh$.  Apply Theorem \ref{thm:Witness} to $\fol_h$ and denote the set satisfying its statement by $X_1$. 

Next, consider the leaves of the foliation induced by $\fol$ on $\Sigma_h$, more precisely we stratify $\Sigma_h$ by a stratification regular with respect to $\dist$.  Note that $h$ is constant on the leaves of this foliation.  We apply to the strata of this stratification 
Proposition \ref{prop:existenceoftheset}. Let $S$ be a stratum from the conclusion of 
Proposition \ref{prop:existenceoftheset}.  It is of dimension $\dim S= r_S\le r$. It is clear that the union of such sets and $X_1$ satisfies the claim of the theorem except (ii) and (iv). 

The point (ii) follows for $c$ small from a general result, the local uniform bound of 
the volume of relatively compact subanalytic sets in subanalytic families, see e.g. \cite[page 261]{hardt82} or  \cite[Th\'eor\`eme 1]{lr98}.  

The transversality of point (iv) follows from (ii) of Theorem \ref{thm:Witness} and the subanalytic Sard theorem applied to the function $h$ restricted to the sets $Y_i$.  The set of critical values, being subanalytic and of measure zero has to be finite.  We choose $\epsilon$ 
smaller that the smallest positive critical value. To have the condition $\partial Y^c = \overline{Y^c} \setminus Y^c \subset Z^c$ we just add $\overline Y\setminus Y$ to Z.  
\qed

\section{Proof of Theorem \ref{thm:Sardminrank}}\label{SECSardminrank}
We follow the notation and we use several results given in \cite[Section 3]{bprfirst}. In particular, recall that the cotagent bundle $T^{\ast}M$ is equipped with a canonical symplectic form $\omega$.

Assume that $M$ (of dimension $n\geq 3$) and $\Delta$ (of rank $m\geq 2$) are analytic and suppose for the sake of contradiction that there is $\bar{x}\in M$ such that the set 
$$
 \mbox{Abn}^{m}_{\Delta}(\bar{x}) = \Bigl\{ \gamma(1) \, \vert \, \gamma \in \Omega_{\Delta}^{\bar{x}} \mbox{ s.t. } \mbox{rank}^{\Delta} (\gamma)=m\Bigr\}
$$
has positive Lebesgue measure in $M$. We equip $M$ with a complete smooth Riemannian metric $g$. Let us now recall the setting provided by \cite[Theorem 1.1]{bprfirst}: there exist a subanalytic Whitney stratification $\mathcal{S}=(\mathcal{S}_{\alpha})$ of $\DeltaPerp$, three subanalytic distributions 
\[
\vec{\mathcal{K}} \subset \vec{\mathcal{J}} \subset \vec{\mathcal{I}} \subset T\DeltaPerp
\]
adapted to $\mathcal{S}$ satisfying properties (i)-(iv). Then, denoting by $\mathcal{S}_0$ the essential domain, that is the union of all strata of $\mathcal{S}$ of maximal dimension, and by $\Sigma$ its complement in $\DeltaPerp$ of dimension strictly less than $2n-m=\dim (\DeltaPerp)$, \cite[Theorem 1.1]{bprfirst} implies that $\mathcal{S}_0$ is an open set in $\DeltaPerp$, $\Sigma$ is an analytic set in $\DeltaPerp$ of codimension at least 1, and $\vec{\mathcal{K}}_{\vert \mathcal{S}_0}= \vec{\mathcal{J}}_{\vert \mathcal{S}_0}=\vec{\mathcal{I}}_{\vert \mathcal{S}_0}$ is isotropic and integrable on $\mathcal{S}_0$ of rank $m_0$ verifying $m_0 \equiv m \, (2)$ and $m_0\leq m-2$. Note, furthermore, that \cite[Proposition 3.6]{bprfirst} combined with the contradiction assumption implies that $m_0>0$, that is, the distribution $\vec{\mathcal{K}}$ yields a non-trivial foliation over $\mathcal{S}_0$ (in particular, $n\geq 4$ and $m \geq 3$). For every $\pa \in \mathcal{S}_{0}$, we denote by $\mathcal{L}_{\pa}\subset \mathcal{S}_{0}$ the leaf of the foliation generated by $\vec{\mathcal{K}}_{\vert \mathcal{S}_0}$ containing $\pa$. 

We start by considering a subset of $\mbox{Abn}^{m}_{\Delta}(\bar{x})$ of positive measure with two extra properties (recall that we have supposed for contradiction that $\mbox{Abn}^{m}_{\Delta}(\bar{x})$ has positive Lebesgue measure in $M$):

\begin{lemma}\label{lem:FirstReduction}
There exist $\bar{\ell} >0$ and a subset $\bar{\mathcal{A}}\subset M$ of positive measure such that, for every point $y\in  \bar{\mathcal{A}}$, the intersection $\pi^{-1}(y) \cap \mathcal{S}_0 \neq \emptyset$ and there exists a singular horizontal curve of minimal rank of length $\leq \bar{\ell}$ (w.r.t $g$) which joins $\bar{x}$ to $y$, for which all abnormal lifts intersect the set $\Sigma$.
\end{lemma}
\begin{proof}[Proof of Lemma \ref{lem:FirstReduction}]
Denote by $\mathcal{A}_{\bar{x}}^{\mathcal{S}_0}$ the set of points $y$ in $\mbox{Abn}^{m}_{\Delta}(\bar{x})$ for which there is a curve $\gamma \in \Omega_{\Delta}^{\bar{x}}$ of minimal rank with $\gamma(1)=y$ which admits an abnormal lift $\psi:[0,1] \rightarrow \DeltaPerp$ such that $\psi([0,1])\subset \mathcal{S}_0$. By construction, the set $\mathcal{A}_{\bar{x}}^{\mathcal{S}_0}$ is contained in the set 
\[
\mbox{ Abn}_{0}(\bar{x}):=  \bigcup_{\pa \in (\mathcal{S}_{0})_{\bar{x}}} \pi \left( \mathcal{L}_{\pa}\right),
\]
so by \cite[Theorem 1.3]{bprfirst} it has Lebesgue measure zero in $M$. We set 
\(
\mathcal{A}_{\bar{x}} := \mbox{Abn}^{m}_{\Delta}(\bar{x}) \setminus \mathcal{A}_{\bar{x}}^{\mathcal{S}_0}
\)
and note that, without loss of generality, we may assume that $\mathcal{A}_{\bar{x}}$ has positive measure in $M$ and that there is $\bar{\ell}>0$ such that for every $y\in \mathcal{A}_{\bar{x}}$ there is a singular horizontal curve of minimal rank of length $\leq \bar{\ell}$ (w.r.t $g$) which joins $\bar{x}$ to $y$ for which all abnormal lifts intersect the set $\Sigma$. Next, recall that $\pi:T^*M\rightarrow M$ denotes the canonical projection and set
\[
\mathcal{A}_{\bar{x}}^{\Sigma} :=\left\{ y \in \mathcal{A}_{\bar{x}} \, \vert \, \pi^{-1}(y) \cap \DeltaPerp \subset \Sigma \right\}.
\]
We observe that the set $\mathcal{A}_{\bar{x}}^{\Sigma}\subset M$ has Lebesgue measure zero in $M$ since otherwise $\Sigma$, which is an analytic set in $\DeltaPerp$ of codimension at least 1 would have positive measure in $\DeltaPerp$. Then, we set
\[
\bar{\mathcal{A}}:=\mathcal{A}_{\bar{x}} \setminus \mathcal{A}^{\Sigma}_{\bar{x}} \subset M,
\]
which by construction has positive Lebesgue measure in $M$.
\end{proof}

We now make a short interlude to introduce three objects which are going to be used in the proof, namely a complete Riemannian metric $\tilde{g}$ over $\DeltaPerp$, locally defined $\vec{\mathcal{K}}_{|\mathcal{S}_0}$-normal forms and transition maps, and a $\vec{\mathcal{K}}_{|\mathcal{S}_0}$-transverse measure.

\medskip
\noindent
\textbf{The metric $\tilde{g}$ over $\DeltaPerp$:} we can extend the Riemannian metric $g$ over $M$ into a complete smooth metric $\tilde{g}$ ``compatible" with $g$ over $\vec{\Delta}$ on $\DeltaPerp$. As a matter of fact, we can define for every $\pa \in \DeltaPerp$,
$$
\tilde{g}_{\pa}(\xi_1,\xi_2) := g_{\pi(\pa)}(d_{\pa}\pi(\xi_1),d_{\pa}\pi(\xi_2)) \qquad \forall \xi_1, \xi_2 \in \vec{\Delta}(\pa),
$$
which is nondegenerate because $\vec{\Delta}$ is always transverse to the vertical fiber of the canonical projection $\pi: T^{\ast}M \to M$, c.f. \cite[Section 3.2]{bprfirst}, and extend $\tilde{g}$ to the missing directions to obtain a complete smooth Riemannian metric on $\DeltaPerp$. In the sequel, we denote by $|\cdot|^{\tilde{g}}$ the norm given by $\tilde{g}$ and by $d^{\tilde{g}}$ the geodesic distance with respect to $\tilde{g}$. Then, we denote by $\mbox{length}^{\tilde{g}}$ the length of an absolutely continuous curve $\psi: [0,1] \rightarrow \DeltaPerp$ with respect to $\tilde{g}$ and note that if $\psi$ is a lift of a singular horizontal path $\gamma:[0,1] \rightarrow M$ then
\[
\mbox{length}^{\tilde{g}}(\psi) = \mbox{length}^g(\gamma).
\]

\medskip
\noindent
\textbf{Local normal form and transition map:} Fix a density point $\bar{y} \in \bar{\mathcal{A}} \setminus \{\bar{x}\}$ together with some $\bar{\pa}\in \mathcal{S}_0$ such that $\pi(\bar{\pa})=\bar{y}$. By considering a local set of coordinates in an open neighborhood $\mathcal{U}\subset M$ of $\bar{y}$, we may assume that we have coordinates $(y,q)$ in $T^*\mathcal{U} = \mathcal{U} \times (\R^n)^*$ in such a way that the restriction of $\pi$ to $T^*\mathcal{U} $ is given by $\pi(y,q)=y$ for all $(y,q)\in T^*\mathcal{U}$. Then, we let $\bar{q}\in T_{\bar{y}}^*\mathcal{U}$ such that $\bar{\pa}=(\bar{y},\bar{q})$, we set 
\(
r:=2n-m-m_0,
\)
and, since $\vec{\mathcal{K}}$ defines a foliation of dimension $m_0>0$ in $\mathcal{S}_0$, as noted in the beginning of the section, we may consider a foliation chart $(\mathcal{W},\varphi)$ of $\bar{\pa}$ such that $\bar{\pa}\in \mathcal{W} \subset \mathcal{S}_0\cap T^*\mathcal{U}$ and for which there are two open sets $W^{1} \subset \R^{m_0}$ and $W^{2} \subset \R^r$ such that $\varphi =(\varphi_1,\varphi_2): \mathcal{W}  \rightarrow  W:=W^{1} \times W^{2}$
is an analytic diffeomorphism satisfying $\bar{a}:=\varphi(\bar{\pa})=0$ and 
\begin{eqnarray}\label{EQ22fev1}
d_{\pa}\varphi \bigl( \vec{\mathcal{K}}(\pa)\bigr) = \vec{K} := \R^{m_0} \times \{0\} \qquad \forall \pa \in \mathcal{W}.
\end{eqnarray}
We note that, by construction, for every $a=(a_1,a_2) \in W$, the plaque $\varphi^{-1}(W^{1}\times \{a_2\})$ is contained in the leaf $\mathcal{L}_{\varphi^{-1}(a)}$ of $\vec{\mathcal{K}}$ in $\mathcal{S}_0$. We also consider a family of local disjoint transverse sections to $\vec{\mathcal{K}}$ in $\mathcal{W}$ parametrized by the connected component of $\mathcal{L}_{\bar{\pa}} \cap \mathcal{W}$ containing $\bar{\pa}$ and given by (see Figure \ref{fig1})
\[
\mathcal{T}_{\pa} := \varphi^{-1} \left( \bigl\{\varphi_1(\pa)\bigr\} \times W^{2}\right) \qquad \forall \pa \in \mathcal{L}_{\bar{\pa}} \cap \mathcal{W}.
\]

\begin{figure}
\begin{center}
\includegraphics[width=9cm]{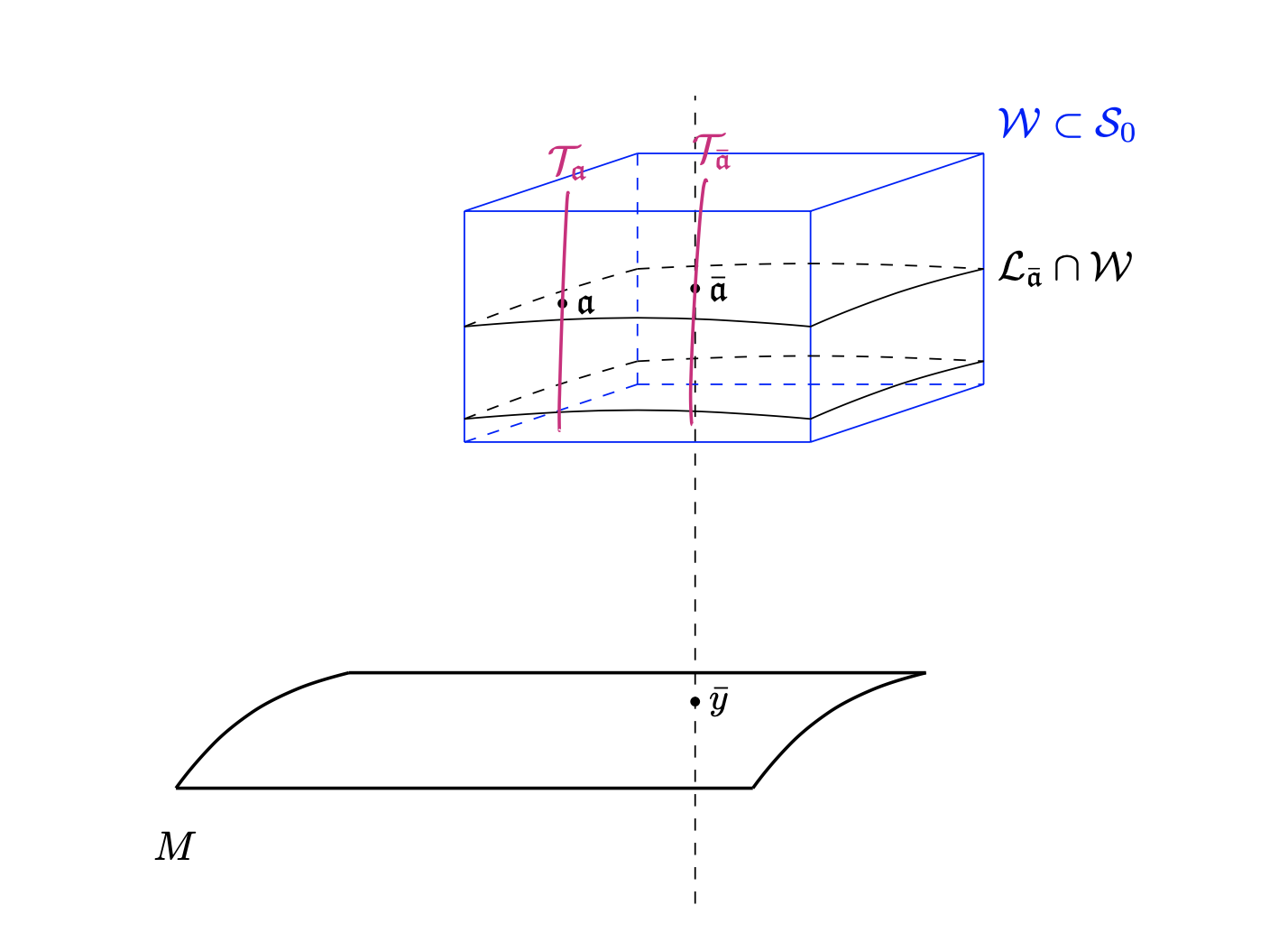}
\caption{Local foliation chart and transverse sections \label{fig1}}
\end{center}
\end{figure}

Up to shrinking $\mathcal{W}$, this family of sections allows us to define a local transition maps parametrized by the connected component of $\mathcal{L}_{\bar{\pa}} \cap \mathcal{W}$ containing $\bar{\pa}$, that is, diffeomorphisms $T^{\bar{\pa},\pa}: \mathcal{T}_{\bar{\pa}} \to \mathcal{T}_{\pa}$ for all $\pa \in \mathcal{L}_{\bar{\pa}} \cap \mathcal{W}$ defined by
\begin{eqnarray}\label{24mars5}
T^{\bar{\pa},\pa} (\pb) := \varphi^{-1} \Bigl( \left\{\varphi_1(\pa)\right\} \times \pi_2\left( \varphi(\pb)\right) \Bigr) \qquad \forall \pb \in \mathcal{T}_{\bar{\pa}}
\end{eqnarray}
Given a subset $\Gamma^{\bar{\pa}}$ of $\mathcal{T}_{\bar{\pa}}$, we will sometimes abuse notation and write
\[
\Gamma^{\bar{\pa}}_{\pa} := T^{\bar{\pa},\pa} (\Gamma^{\bar{\pa}}).
\]

\medskip
\noindent
\textbf{Transverse metric:} We define a $2l$-form $\eta$ on $\DeltaPerp$ by
\[
\eta:=\bigl(\omega^{\perp}\bigr)^l \quad \mbox{with} \quad l:= \frac{r}{2},
\]
where $r$ is the co-rank of $\vec{\mathcal{K}}|_{\mathcal{S}_0}$ in respect to $\DeltaPerp$, that is, $r = 2n-m - m_0$. The following lemma follows essentially from \cite[Proposition 3.2(ii)]{bprfirst} and the assumption that $\vec{\mathcal{K}}_{\vert S_0}$ is splittable, cf. $\S$ \ref{ssec:splittable}. 

\begin{lemma}\label{lem:SecondReduction}
There are a $\vec{\mathcal{K}}$-transverse section $\mathcal{T}_{\bar{\pa}}\subset \mathcal{S}_0$ centered at $\bar{\pa}$ and a compact set $\tilde{\mathcal{A}}^{\bar{\pa}} \subset \mathcal{T}_{\bar{\pa}}$ such that  the following properties are satisfied:
\begin{itemize}
\item[(i)] The set $\tilde{\mathcal{A}}^{\bar{\pa}}$ has positive measure with respect to the volume form $\eta_{\vert \mathcal{T}_{\bar{\pa}}}$.
\item[(ii)] For every $\pa \in \tilde{\mathcal{A}}^{\bar{\pa}}$, there is an absolutely continuous curve $\psi:[0,1] \rightarrow \DeltaPerp$ such that
\[
\psi(0)=\pa, \quad \psi(1) \in \Sigma, \quad \mbox{\em length}^{\tilde{g}}(\psi) \leq \bar{\ell}+1
\quad \mbox{and} \quad \psi(t) \in \mathcal{L}_{\pa} \subset \mathcal{S}_0 \quad \forall t \in [0,1).
\]
\item[(iii)] For any distinct points $\pa, \pa' \in \tilde{\mathcal{A}}^{\bar{\pa}}$, $\pa$ and $\pa'$ are not $(\vec{\mathcal{K}},2\bar{\ell}+5)$-related.
\end{itemize}
\end{lemma}

\begin{proof}[Proof of Lemma \ref{lem:SecondReduction}]
Recall that $\bar{y}$ is a density point of $\bar{\mathcal{A}} \setminus \{\bar{x}\}$ and $\bar{\pa}\in \mathcal{S}_0$ satisfies $\pi(\bar{\pa})=\bar{y}$. Consider the notation of the local normal form above and for every $y\in \mathcal{U}$, denote by $\mathcal{V}_{y}$ the vertical fiber in $\DeltaPerp$ over $y$ given by 
\[
\mathcal{V}_{y}:=\pi^{-1}(\{y\}) \cap T_{\bar{\pa}} \DeltaPerp,
\]
which coincides with a vector space $\vec{V}_y$ of dimension $n-m$ with the origin removed. Since $\vec{\mathcal{K}}(\bar{\pa})\cap T_{\bar{\pa}}\mathcal{V}_{\bar{y}}=\{0\}$ (see \cite[Theorem 1.1(i) and Equation (3.5)]{bprfirst}), there is a vector space $\vec{\mathcal{H}}\subset T_{\bar{\pa}}\DeltaPerp$ of dimension $n$ containing $\vec{\mathcal{K}} (\bar{\pa})$ which is transverse to $T_{\bar{\pa}}\mathcal{V}_{\bar{y}}=\vec{V}_{\bar{y}}$ in $T_{\bar{\pa}}\DeltaPerp$, that is, such that 
\[
\vec{\mathcal{H}} \oplus \vec{V}_{\bar{y}} = T_{\bar{\pa}}\DeltaPerp \quad \mbox{and} \quad \vec{\mathcal{K}} (\bar{\pa}) \subset \vec{\mathcal{H}}.
\]
Then, we consider a vector space $\vec{\mathcal{P}} \subset \vec{\mathcal{H}}$ such that
\begin{eqnarray}\label{EQ22fev2}
\vec{\mathcal{K}} (\bar{\pa}) \oplus \vec{\mathcal{P}} = \vec{\mathcal{H}}
\end{eqnarray}
and define the vector spaces $\vec{\mathcal{Q}} \subset T_{\bar{\pa}}\DeltaPerp$ and $\vec{Q}, \vec{H}, \vec{P} \subset \R^{2n-m}$ by
\[
\vec{\mathcal{Q}} := \vec{\mathcal{P}} \oplus \vec{V}_{\bar{y}},  \quad \vec{Q} := d_{\bar{\pa}}\varphi \bigl( \vec{\mathcal{Q}}\bigr), \quad \vec{H} := d_{\bar{\pa}}\varphi \bigl( \vec{\mathcal{H}}\bigr), \quad \vec{P} := d_{\bar{\pa}}\varphi \bigl( \vec{\mathcal{P}}\bigr).
\]
By construction, $\vec{\mathcal{H}}$ and $\vec{H}$ have dimension $n$, $\vec{\mathcal{P}}$ and $\vec{P}$ have dimension $n-m_0$, $\vec{\mathcal{Q}}$ and $\vec{Q}$ have dimension $r=2n-m-m_0$ and, remembering (\ref{EQ22fev1})-(\ref{EQ22fev2}), we have 
\begin{eqnarray}\label{EQ22fev4}
\vec{K}\oplus\vec{P}= \vec{H}, \quad 
\vec{\mathcal{K}}(\bar{\pa}) \oplus \vec{\mathcal{Q}} = T_{\bar{\pa}}\DeltaPerp, \quad \vec{K}\oplus \vec{Q} = \R^{2n-m}.
\end{eqnarray}
Then, we define two $n$-dimensional open smooth manifolds $H\subset W$ and $\mathcal{H} \subset \mathcal{W}$ by 
\[
H := \vec{H} \cap W \quad \mbox{and} \quad \mathcal{H}:=\varphi^{-1}(H),
\]
and note that the restriction of $\pi$ to $\mathcal{H}$  is a submersion at $\bar{\pa}$. Therefore,  there is a smooth submanifold $I$ of $W$ of dimension $n$ containing $\bar{a}=0$ of the form ($|\cdot|$ stands for the Euclidean norm in $\R^{m_0}$ or $\R^{2n-m}$)
\[
I = \left\{(a_1,0) + p \, \vert \, a_1 \in W^1, \, p\in \vec{P}, \, |a_1| < \delta, |p| < \delta \right\} \subset H,
\]
with $\delta>0$, such that  the mapping 
\[
F =\pi_{\vert \mathcal{I}}\, : \, \mathcal{I}:= \varphi^{-1}(I) \, \longrightarrow \, \mathcal{E} := \pi (\mathcal{I}) 
\]
 is a smooth diffeomorphism. 
 
 \begin{figure}[H]
\begin{center}
\includegraphics[width=6cm]{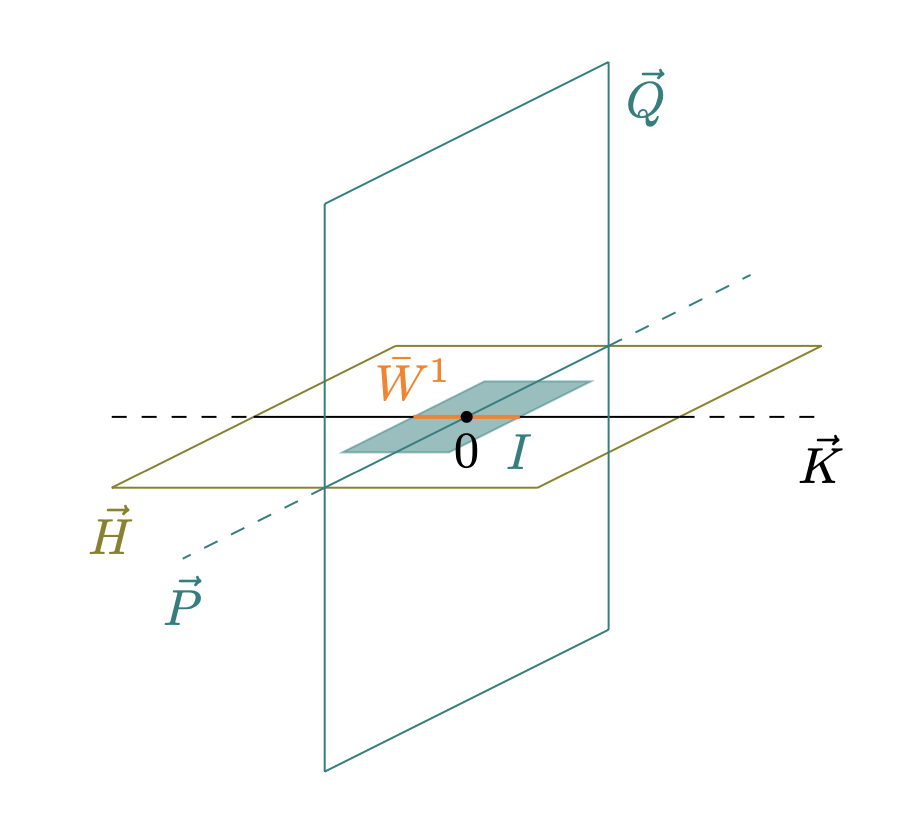}
\caption{A picture showing the sets $\vec{K}, \vec{H}, \vec{P}, \vec{Q}, I$ and $\bar{W}^1$ \label{figrev1}}
\end{center}
\end{figure}
 
Then, we denote by $\bar{W}^1$ the set of $a_1\in W^1$ with $|a_1|\leq \delta /2$ and for each $a_1 \in \bar{W}_1$, we define the sets $P_{a_1} \subset H$, $\mathcal{P}_{a_1}\subset \mathcal{H}$, $\mathcal{E}_{a_1}\subset \mathcal{E}$, $\mathcal{Q}_{a_1}\subset \mathcal{W}$ and $Q_{a_1}\subset W$ by 
\[
P_{a_1} := \bigl((a_1,0) + \vec{P}\bigr) \cap I, \quad \mathcal{P}_{a_1} := \varphi^{-1} \left(P_{a_1}\right), \quad \mathcal{E}_{a_1} := F \left(  \mathcal{P}_{a_1}\right).
\] 
\[
\mathcal{Q}_{a_1} :=\left\{ (y,q) + (0,h) \, \vert \, (y,q) \in \mathcal{P}_{a_1}, \, h \in \vec{V}_y, \, |h| < \delta\right\}  \quad \mbox{and} \quad Q_{a_1} := \varphi\left( \mathcal{Q}_{a_1} \right).
\]
By construction, for each $a_1 \in \bar{W}^1$, the set $P_{a_1}$ is an open smooth submanifold of $W$ of dimension $n-m_0$, the set $\mathcal{P}_{a_1}$ is an open smooth submanifold of $\mathcal{W}$ of dimension $n-m_0$ and  the set $\mathcal{E}_{a_1}$ is an open smooth submanifold of $\mathcal{U}$ of dimension $n-m_0$. In fact, the collection $\{\mathcal{P}_{a_1}\}_{a_1\in \bar{W}_1}$ defines a collection of pairwise disjoint slices (or plaques to use the terminology of foliations recalled above) that cover $\mathcal{I}\subset \mathcal{H}$ and projects to the collection of pairwise disjoint slices $\{\mathcal{E}_{a_1}\}_{a_1\in \bar{W}_1}$ covering $\mathcal{E}=\pi(\mathcal{I}) \subset M$. Furthermore, since by $\vec{\mathcal{P}} \cap \vec{V}_{\bar{y}} = d_0\varphi^{-1} (\vec{P}) \cap T_{\bar{\pa}}\mathcal{V}_{\bar{y}}=\{0\}$ the mapping 
\[
(\pa,h) \in \left\{ ((y,q),h) \, \vert \, (y,q) \in \mathcal{P}_0, \, h \in \vec{V}_y \right\} \, \longmapsto \, \pa + (0,h) \in \DeltaPerp
\]
 is an immersion at $(\bar{\pa},0)$ valued in $\mathcal{Q}_0$ and since the mapping 
 \[
a \in \vec{Q} \, \longmapsto \bigl(0,\pi^2(a)\bigr)\in T_0 \bigl(\{0\}\times W^2\bigr)
 \]
is a linear isomorphism (by (\ref{EQ22fev4}) we have $\vec{K}\oplus \vec{Q} = \R^{2n-m}$), we may assume by taking $\delta>0$ small enough that, for each $a_1 \in \bar{W}^1$, the sets $\mathcal{Q}_{a_1} $ and $Q_{a_1}$ are open smooth manifolds of dimension $r$ and that the mapping
\[
G_{a_1} \, : \, a \in  Q_{a_1} \, \longmapsto \, \bigl(0,\pi^2(a)\bigr)\in T_0 \bigl(\{0\}\times W^2\bigr)
\]
is a smooth diffeomorphism from $Q_{a_1}$ onto its image $G_{a_1}(Q_{a_1})$. Thus, $\{\mathcal{Q}_{a_1}\}_{a_1\in \bar{W}_1}$ defines a collection of pairwise disjoint slices that extends the collection $\{\mathcal{P}_{a_1}\}_{a_1\in \bar{W}_1}$ (for each $a_1 \in \bar{W}^1$, $\mathcal{P}_{a_1}\subset \mathcal{Q}_{a_1}$) and covers a neigborhood of $\bar{\pa}$ in $\Delta^{\perp}$. Furthermore, we observe that for every  $a_1 \in \bar{W}^1$ and every $a\in Q_{a_1}$, the two points $a$ and $b=G_{a_1}(a)$ have the same coordinate in $W^2$ so that their images by $\varphi^{-1}$, $\varphi^{-1}(a)$ and $\varphi^{-1}(b)$, belong to the same plaque and to the same leaf of the foliation defined by $\vec{\mathcal{K}}$ in $\mathcal{W}$, so, by the construction made before the statement of the lemma,  the points $\varphi^{-1}(a)\in \mathcal{W}$ and $\varphi^{-1}(b)\in \mathcal{T}_{\bar{\pa}}$ can be connected through a smooth curve horizontal with respect to $\vec{\mathcal{K}}$ of length (w.r.t $\tilde{g}$) less than 1.  In other words, for every $a_1 \in \bar{W}^1$, the mapping $ \varphi^{-1} \circ G_{a_1} \circ \varphi$ acts as a projection from the slice $\mathcal{Q}_{a_1}$ to $\mathcal{T}_{\bar{\pa}}$ along horizontal curves with respect to $\vec{\mathcal{K}}$ (and with length less than 1). We are now ready to conclude the proof of the Lemma which consists in applying Fubini's Theorem to select a slice $\mathcal{E}_{\tilde{a}_1}$ in $\{\mathcal{E}_{a_1}\}_{a_1\in \bar{W}_1}$ whose intersection with $\bar{\mathcal{A}}$ has positive $(n-m_0)$-dimensional Lebesgue measure, to lift the intersection to $\mathcal{Q}_{\tilde{a}_1}\subset \Delta^{\perp}$, and to project it to $\mathcal{T}_{\bar{\pa}}$ by action of $\varphi^{-1} \circ G_{\tilde{a}_1} \circ \varphi$.

By construction, the sets $P_{a_1}$ as well as $\mathcal{P}_{a_1}$, $\mathcal{E}_{a_1}$, with $a_1 \in \bar{W}^1$
are pairwise disjoint and satisfy
\[
 \bigcup_{a_1\in \bar{W}^1} P_{a_1} = I \quad \bigcup_{a_1\in \bar{W}^1} \mathcal{P}_{a_1} = \mathcal{I},  \quad \bigcup_{a_1\in \bar{W}^1} \mathcal{E}_{a_1} = \mathcal{E}.
\]
Since $\bar{y}$ is a density point of $\bar{\mathcal{A}}$, by Fubini's Theorem, we infer that there is $\tilde{a}_1 \in \bar{W}^1$ such that the $(n-m_0)$-dimensional Lebesgue measure of the set
\[
\Theta := \bar{\mathcal{A}} \cap \mathcal{E}_{\tilde{a}_1} \subset \mathcal{E}_{\tilde{a}_1}
\]
is positive. In fact, by taking a compact subset of $\Theta$ of positive measure, we may indeed assume that $\Theta$ is compact. By construction, for every $\theta \in \Theta$, there is an horizontal path  $\gamma^{\theta} \in \Omega_{\Delta}^{\bar{x}}$ of length $\leq \bar{\ell}$ (w.r.t $g$) such that $\gamma^{\theta}(1)=\theta$, $\mbox{rank}^{\Delta} (\gamma)=m$ and  for which all abnormal lifts meet the set $\Sigma$. Hence, by \cite[Proposition 3.4]{bprfirst}, for every $p\in \DeltaPerp_{\theta}$, $\gamma^{\theta}$ admits an abnormal lift $\psi^{\theta,p}:[0,1]\rightarrow \DeltaPerp$ such that $\psi^{\theta,p}(1)=(\theta,p)$ and $\psi^{x,p}([0,1])\cap \Sigma \neq \emptyset$. Thus, we obtain that any $\pa$ in the set 
\[
\tilde{\Theta} := \left\{ F^{-1} (\theta) + (0,h) \, \vert \, \theta \in \Theta, \, h\in \vec{V}_{\theta}, \, |h| \leq \delta \right\}
\]
can be joined to $\Sigma$ by a curve of length $\leq \bar{\ell}$. By Fubini's Theorem, the set $\tilde{\Theta}$ is a compact set of positive measure in the manifold $\mathcal{Q}_{\tilde{a_1}}$, thus its image by $\varphi$, $\varphi(\tilde{\Theta})$, is a compact set of positive measure in the manifold $Q_{\tilde{a_1}}$, the image of $\varphi(\tilde{\Theta})$ by $G_{\tilde{a}_1}$, $\Lambda :=(G_{\tilde{a}_1}\circ \varphi)(\tilde{\Theta})$ has positive measure in $\{0\} \times W^2$ and the image of $\Lambda$ by $\varphi^{-1}$ has positive measure in $\mathcal{T}_{\bar{\pa}}$. By construction, any point of $\varphi^{-1}(\Lambda)$ can be joined to a point of $\Sigma$ by an absolutely continuous curve horizontal with respect to $\vec{\mathcal{K}}$ of length (w.r.t $\tilde{g}$) $\leq \bar{l}+1$. By assumption of splittability, we can select in $\varphi^{-1}(\Lambda)$ a compact subset $\tilde{\mathcal{A}}^{\bar{\pa}}$ of positive measure satisfying the same property and whose points are not $(\vec{\mathcal{K}},2\bar{\ell}+5)$-related. We complete the proof by applying \cite[Proposition 3.2(ii)]{bprfirst}.
\end{proof}

The next result combines the geometrical framework of Lemma \ref{lem:SecondReduction} with a compactness argument and the witness section given by Theorem \ref{thm:CorWitness}.

\begin{lemma}\label{lem:WitnessReduction}
There are a point $\hat{\pa}\in \Sigma$, a compact set $\check{\mathcal{A}}^{\bar{\pa}}\subset \tilde{\mathcal{A}}^{\bar{\pa}} \subset \mathcal{T}_{\bar{\pa}}$, a relatively compact open neighborhood $V\subset \DeltaPerp$ of $\hat{\pa}$, a compact set $\check{\Sigma} \subset \Sigma \cap V$, a real analytic function $h: V \to [0,\infty)$, a \ssanalytic set $X\subset V\setminus \Sigma$ and $C, \nu, \epsilon>0$ such that the following properties are satisfied:
\begin{itemize}
\item[(i)] The set $h^{-1}(0)$ is equal to $\Sigma \cap V$.
\item[(ii)] For every $0<c<\epsilon$, the \ssanalytic set $X^c:=X\cap h^{-1}(c)$ has $r$-dimensional volume with respect to $\tilde{g}$ bounded by $C$. In particular, $X^c$ is a $r$-dimensional set and $X$ is a $(r+1)$-dimensional set.

\item[(iii)] For every $0<c<\epsilon$ and for every $\pa \in h^{-1}(c) \subset  V \setminus \Sigma$, there is a smooth curve $\alpha :[0,1] \rightarrow V\setminus \Sigma$ which is contained in $\mathcal{L}_{\pa} \cap h^{-1}(c)$ such that
\[
\alpha(0)=\pa, \quad \alpha(1)\in X^c, \quad \mbox{and} \quad \mbox{\em length}^{\tilde{g}}(\alpha) \leq 1.
\]

\item[(iv)] For every $0<c<\epsilon$, we can decompose $X^c$ as the union of two disjoint \ssanalytic sets $Y^c$ and $Z^c$, such that $Z^c$ has dimension $<r$, and $Y^c$ is the union of finitely many smooth \ssanalytic sets $Y_i^c$, with $i\in I^c$, of dimension $r$ such that 
\[
\partial Y^c = \overline{Y^c} \setminus Y^c \subset Z^c \quad \text{ and } \quad  T_{\pa} \DeltaPerp = T_{\pa}Y_i^c + \vec{\mathcal{K}}(\pa) \qquad \forall \pa \in Y_i^c, \, \forall i \in I_c.
\]

\item[(v)] For all $\pa \in \check{\mathcal{A}}^{\bar{\pa}}$, there is an absolutely continuous curve $\psi:[0,1] \rightarrow \DeltaPerp$ such that
\[
\psi(0)=\pa, \quad \psi(1) \in \check{\Sigma}, \quad \mbox{\em length}^{\tilde{g}}(\psi) \leq \bar{\ell}+1
\quad \mbox{and} \quad \psi(t) \in \mathcal{L}_{\pa} \subset \mathcal{S}_0 \quad \forall t \in [0,1).
\]
\item[(vi)] The set $\check{\mathcal{A}}^{\bar{\pa}}$ has measure $\geq \nu$ with respect to the volume form $\eta_{\vert \mathcal{T}_{\bar{\pa}}}$.
\end{itemize}
Moreover, there is a continuous function $\delta:[0,\infty) \rightarrow [0,\infty)$ with $\delta(0)=0$ such that for every $0<c<\epsilon$ and every $\pa \in h^{-1}(c)$,
 \begin{eqnarray}\label{23mars1}
\left| \eta_{\pa} \left(\xi_1, \ldots, \xi_{d}\right) \right| \leq \delta(c)|\xi_1|^{\tilde{g}} \cdots |\xi_d|^{\tilde{g}} \qquad \forall \xi_1, \ldots, \xi_{d} \in  T_{\pa} \DeltaPerp.
 \end{eqnarray}
\end{lemma}
\begin{proof}[Proof of Lemma \ref{lem:WitnessReduction}]
Let $N:=\DeltaPerp$ be the real-analytic manifold of dimension $2n-m$ equipped with the singular analytic foliation $\mathcal{F}$ of generic corank $r=2n-m-m_0$ with singular set $\Sigma$ and $\mathcal{B} \subset \Sigma$ the set of $\pa \in \Sigma$ for which there is an absolutely continuous curve $\psi:[0,1] \rightarrow \DeltaPerp$ such that
\begin{eqnarray*}
\psi(0)\in \tilde{\mathcal{A}}^{\bar{\pa}}, \quad \psi(1)=\pa,  \quad \mbox{length}^{\tilde{g}}(\psi) \leq \bar{\ell}+1 \quad \mbox{and} \quad \psi(t) \in \mathcal{L}_{\psi(0)} \subset \mathcal{S}_0 \quad \forall t \in [0,1),
\end{eqnarray*}
where $\tilde{\mathcal{A}}^{\bar{\pa}}\subset \mathcal{T}_{\bar{\pa}}$  is the set provided by Lemma \ref{lem:SecondReduction}. The compactness of $\tilde{\mathcal{A}}^{\bar{\pa}}$ together with the closedness of $\Sigma$ and the upper bound on the length of curves (with the completeness of $\tilde{g}$) imply that $\mathcal{B}$ is a compact subset of $\Sigma$. By Theorem  \ref{thm:CorWitness} applied with $\ell=1$, for every $\pa \in \mathcal{B}$, there are a relatively compact open neighborhood $V_{\pa}$ of $\pa$ in $N=\DeltaPerp$, a real-analytic function $h_{\pa}: V_{\pa} \to [0,\infty)$, a \ssanalytic set $X_{\pa}\subset V_{\pa}\setminus \Sigma$ and $C_{\pa}>0$ such that the properties (i)-(iv) of Theorem  \ref{thm:CorWitness} are satisfied. Pick for each $\pa \in \mathcal{B}$ a compact neighborhood $\check{V}_{\pa}\subset V_{\pa}$ of $\pa$ and consider by compactness of $\mathcal{B}$ a finite family $\{\pa_i\}_{i\in I}$ such that
\begin{eqnarray}\label{23mars2}
\mathcal{B} \subset \bigcup_{i\in I} \check{V}_{\pa_i}  \subset \bigcup_{i\in I} V_{\pa_i}.
\end{eqnarray}
Then, for every $i\in I$, denote by $\tilde{\mathcal{A}}_i^{\bar{\pa}}$ the set of $\pa\in \tilde{\mathcal{A}}^{\bar{\pa}}$ for which  there is an absolutely continuous curve $\psi:[0,1] \rightarrow \DeltaPerp$ such that
\begin{eqnarray*}
\psi(0)= \pa, \quad \psi(1)\in \check{\Sigma}_i,  \quad \mbox{length}^{\tilde{g}}(\psi) \leq \bar{\ell}+1 \quad \mbox{and} \quad \psi(t) \in \mathcal{L}_{\pa} \subset \mathcal{S}_0 \quad \forall t \in [0,1),
\end{eqnarray*}
with $\check{\Sigma}_i:= \Sigma \cap \check{V}_{\pa_i}\cap \mathcal{B}$. We claim that each set $\tilde{\mathcal{A}}_i^{\bar{\pa}}$ is a Borel subset of $\tilde{\mathcal{A}}^{\bar{\pa}}$. As a matter of fact, for each $i\in I$, we can write
\[
\tilde{\mathcal{A}}_i^{\bar{\pa}} = \bigcap_{k\in \N^*} \tilde{\mathcal{A}}_{i,k}^{\bar{\pa}}, 
\]
where for each $k\in \N^*$, the set $\tilde{\mathcal{A}}_{i,k}^{\bar{\pa}}$ is defined as the set of $\pa\in \tilde{\mathcal{A}}^{\bar{\pa}}$ for which  there is an absolutely continuous curve $\psi:[0,1] \rightarrow \mathcal{L}_{\pa} \in \mathcal{S}_0$ such that
\begin{eqnarray*}
\psi(0)= \pa, \quad \psi(1)\in B_{1/k}^{\tilde{g}} \left(\check{\Sigma}_i\right) \cap \left( V\setminus \Sigma\right) ,  \quad \mbox{length}^{\tilde{g}}(\psi) < \bar{\ell}+1,
\end{eqnarray*}
with 
\[
B_{1/k}^{\tilde{g}} \left(\check{\Sigma}_i\right) := \left\{ \pa' \in \DeltaPerp \, \vert \, d^{\tilde{g}}\left(\pa',\check{\Sigma}_i\right) < \frac{1}{k} \right\}.
\]
By regularity of $\vec{\mathcal{K}}_{|\mathcal{S}_0}$, each set $\tilde{\mathcal{A}}_{i,k}^{\bar{\pa}}$ is open in $\tilde{\mathcal{A}}^{\bar{\pa}}$, so we infer that each $\tilde{\mathcal{A}}_i^{\bar{\pa}}$ is a Borel subset of $\tilde{\mathcal{A}}^{\bar{\pa}}$. Furthermore, by construction of $\mathcal{B}$, (\ref{23mars2}) and Lemma \ref{lem:SecondReduction} (ii), we have
\[
\tilde{\mathcal{A}}^{\bar{\pa}} = \bigcup_{i\in I} \tilde{\mathcal{A}}_i^{\bar{\pa}}.
\]
As a consequence, since $\tilde{\mathcal{A}}^{\bar{\pa}}$ has positive measure with respect to the volume form $\eta_{\vert \mathcal{T}_{\bar{\pa}}}$ (Lemma  \ref{lem:SecondReduction} (i)), there is $i\in I$ such that $\tilde{\mathcal{A}}^{\bar{\pa}}_i$ and a compact subset  $\check{\mathcal{A}}^{\bar{\pa}}_i$ of it satisfy the same property.  We conclude the proof of (i)-(vi) by setting $\check{\mathcal{A}}^{\bar{\pa}}:= \check{\mathcal{A}}^{\bar{\pa}}_i$, $V:=V_{\pa_i}$, $\check{\Sigma} := \check{\Sigma}_i$, $h:=h_{\pa_i}$, $X:=X_{\pa_i}$, $C:=C_{\pa_i}$ and $\nu$ the volume of $\check{\mathcal{A}}^{\bar{\pa}}$ with respect to $\eta_{\vert \mathcal{T}_{\bar{\pa}}}$.

The second part of the proof (\ref{23mars1}) follows from  \cite[Proposition 3.2(ii)]{bprfirst}. By \cite[Theorem 1.1(iv)]{bprfirst}, we have
\[
\dim \left( \mbox{ker} \bigl( \omega^{\perp}_{\pa} \bigr) \right) \geq m_0 +2 \qquad \forall \pa \in \Sigma.
\]
Therefore, by \cite[Proposition 3.2(iii)]{bprfirst}, we have $\eta_{\pa}=0$ for all $\pa\in \check{\Sigma}$ and we can conclude by regularity of $h$ near the compact set $\check{\Sigma} \subset V$.
\end{proof}

The idea of our proof consists now in obtaining a contradiction from the construction of an homotopy sending smoothly the points of a small neighborhood of a set $\check{\mathcal{A}}^{\bar{\pa},c} \subset \check{\mathcal{A}}^{\bar{\pa}}$ in $\mathcal{T}_{\bar{\pa}}$ to an open subset of $Y^c$ for $c>0$ small enough. Since this homotopy has to preserve the leaves of $\vec{\mathcal{K}}_{|\mathcal{S}_0}$, we perform the construction by following the minimizing geodesics from $\mathcal{T}_{\bar{\pa}}$ to $Y^c$ with respect to some complete metric on $\mathcal{S}_0$ that needs to be built (note that $\tilde{g}$ is not complete when restricted to $\mathcal{S}_0$). The next Lemma formalizes this framework:

\begin{lemma}\label{lem:MetricReduction}
For every $0<c<\epsilon$, there are a smooth Riemannian metric $\tilde{g}^c$ on $\mathcal{S}_0$ and a compact set $\check{\mathcal{A}}^{\bar{\pa},c}\subset \check{\mathcal{A}}^{\bar{\pa}}$  satisfying the following properties:
\begin{itemize}
\item[(i)] The Riemannian manifold $(\mathcal{S}_0,\tilde{g}^c)$ is complete.

\item[(ii)] For every $\pa\in \check{\mathcal{A}}^{\bar{\pa},c}$, there  is an absolutely continuous curve $\psi:[0,1] \rightarrow \mathcal{L}_{\pa} \subset \mathcal{S}_0$ such that (where $Y^c$ is defined in Lemma \ref{lem:WitnessReduction}(iv))
\[
\psi(0)=\pa, \quad \psi(1) \in Y^c \quad \mbox{and} \quad \mbox{\em length}^{\tilde{g}}(\psi) = \mbox{\em length}^{\tilde{g}^c}(\psi) < \bar{\ell}+2.
\]

\item[(iii)] The set $\check{\mathcal{A}}^{\bar{\pa},c}$ has measure $\geq \nu/4$ with respect to the volume form $\eta_{\vert \mathcal{T}_{\bar{\pa}}}$.

\item[(iv)] Let $\mathcal{C}^{c}\subset \mathcal{S}_0$ be the set of points $\pa \in \mathcal{S}_0$ for which there is an absolutely continuous curve $\psi:[0,1]\rightarrow \mathcal{L}_{\pa}$ of length $\leq \bar{\ell}+2$ with respect to $\tilde{g}^c$ joining $\pa$ to a point of $\overline{Z^c}$ (defined in Lemma \ref{lem:WitnessReduction}(iv)). Then $\mathcal{C}^{c}$ is closed and does not intersect $\check{\mathcal{A}}^{\bar{\pa},c}$.
\end{itemize}
\end{lemma}

\begin{proof}[Proof of Lemma \ref{lem:MetricReduction}]
Since $\Sigma$ is a closed subset of $\DeltaPerp$, we can pick a smooth function  $F:\DeltaPerp \rightarrow [0,\infty)$ such that
\[
\Sigma = F^{-1} \left(\{0\}\right)
\]
and fix some $c>0$. Consider the function $D:\R^+ \rightarrow [0,+\infty]$ given  by
\[
D(\lambda ) := \left\{ 
\begin{array}{rcl}
\frac{1}{\bar{\ell} + 2 -\lambda} & \mbox{ if } &\lambda < \bar{\ell}+2 \\
+\infty & \mbox{ if } &\lambda \geq \bar{\ell} + 2
\end{array}
\right.
\qquad \forall \lambda \in \R^+
\]
and define the function $\Psi^c : \mathcal{T}_{\bar{\pa}} \rightarrow [0,\infty]$ by
\[
\Psi^c (\pa) := \inf \Bigl\{ D \left( \mbox{length}^{\tilde{g}}(\psi) \right) + \max \left(F\bigl(\psi([0,1])\bigr)^{-1}\right) \, \vert \, \psi \in \Omega\left(\pa,\overline{X^c}\right)\Bigr\}
\]
where for every $\pa \in \mathcal{T}_{\bar{\pa}}$, $\Omega\left(\pa,\overline{X^c}\right)$ stands for the set of absolutely continuous curves $\psi:[0,1] \rightarrow \mathcal{S}_0$ such that $\psi(0)=\pa$, $\psi(1)\in \overline{X^c}$ and $\psi$ is almost everywhere tangent to $\vec{\mathcal{K}}_{|\mathcal{S}_0}$. Let $\pa \in \check{\mathcal{A}}^{\bar{\pa}}$ be fixed, by Lemma \ref{lem:WitnessReduction} (v), there is an absolutely continuous curve $\psi:[0,1] \rightarrow \DeltaPerp$ such that $\psi(0)=\pa$, $\psi(1) \in \check{\Sigma}$, $\mbox{length}^{\tilde{g}}(\psi) \leq \bar{\ell}+1$ and $\psi(t) \in \mathcal{L}_{\pa} \subset \mathcal{S}_0$ for all $t \in [0,1)$. Thus, since $\psi(t)$ belongs to $V\setminus \Sigma$ for $t$ close to $1$, Lemma \ref{lem:WitnessReduction}(iii) shows that $\pa$ can be joined to $X^c$ by a curve tangent to $\mathcal{L}_{\pa}$ contained in $\mathcal{S}_0$ of length  $< \bar{\ell}+2$. Therefore $\Psi^c(\pa)$ is finite for every $\pa \in \check{\mathcal{A}}^{\bar{\pa}}$. Moreover, the function $\Psi^c$ is lower semi-continuous on $ \mathcal{T}_{\bar{\pa}}$ (because we consider curves satisfying $\psi(1)\in \overline{X^c}$ and we may use foliation charts along $\psi$), so we have 
\[
\check{\mathcal{A}}^{\bar{\pa}} = \bigcup_{k\in \N} \check{\mathcal{A}}^{\bar{\pa}}_k \quad \mbox{with} \quad \check{\mathcal{A}}^{\bar{\pa}}_k :=  \left( \Psi^c\right)^{-1}\bigl([0,k]\bigr) \cap \check{\mathcal{A}}^{\bar{\pa}},
\]
where each set of the above union is a compact subset of $\check{\mathcal{A}}^{\bar{\pa}}$. Thus, there is $k\in \N$ such that the measure of $\check{\mathcal{A}}^{\bar{\pa}}_k$ with respect to the volume form $\eta_{\vert \mathcal{T}_{\bar{\pa}}}$ is $\geq \nu/2$ and such that for every $\pa\in \check{\mathcal{A}}^{\bar{\pa}}_k$, there is an absolutely continuous curve $\psi:[0,1] \rightarrow \mathcal{L}_{\pa} \subset \mathcal{S}_0$  satisfying $\psi(0)=\pa$, $\psi(1) \in \overline{X^c}$ and 
\[
D \left( \mbox{length}^{\tilde{g}}(\psi) \right) + \max \left(F\bigl(\psi([0,1])\bigr)^{-1}\right) \leq k,
\]
which implies 
\[
 \mbox{length}^{\tilde{g}}(\psi) < \bar{\ell} +2 \quad \mbox{and} \quad \min \left( F\bigl(\psi([0,1])\bigr) \right) \geq \frac{1}{k}.
\]
Fix a smooth complete metric $\tilde{g}^c$ on $\mathcal{S}_0$ which coincides with $\tilde{g}$ on the set $F^{-1}([1/(k),\infty))$. Recall that the definition of $Y^c$ and $Z^c$ is given in Lemma \ref{lem:WitnessReduction}(iv), and note that $\overline{X^c} = \overline{Y^c} \cup \overline{Z^c}$, where $\overline{Z^c}$ has dimension $\leq r-1$. By Lemma \ref{lem:WitnessReduction} (iv), the boundary $\partial Y^c:=\overline{Y^c}\setminus Y^c$ is contained in $Z^c$, so the set of points of $\mathcal{S}_0$ that can be joined to $\overline{Z^c}$ along absolutely continuous curves tangent to $\vec{\mathcal{K}}_{|\mathcal{S}_0}$ is a countable union of smooth submanifolds of dimension at most $r-1+m_0=2n-m-1$, so it has measure zero in $\mathcal{S}_0$ and in fact since it is invariant by the foliation associated with  $\vec{\mathcal{K}}_{|\mathcal{S}_0}$, its intersection with $\mathcal{T}_{\bar{\pa}}$ has measure zero in $\mathcal{T}_{\bar{\pa}}$ (by Fubini's Theorem). Thus, we can consider a compact subset $\check{\mathcal{A}}^{\bar{\pa},c} $ of  $ \check{\mathcal{A}}^{\bar{\pa}}_k \subset \check{\mathcal{A}}^{\bar{\pa}}$ of measure $\geq \nu/4$ such that the properties (i)-(iii) are satisfied. Finally, the set $\mathcal{C}^c$ is closed because $\partial Y^c$ is closed and $(\mathcal{S}_0,\tilde{g}^c)$ is complete. We conclude that $\check{\mathcal{A}}^{\bar{\pa},c} $ satisfies (iv) by construction.
\end{proof}

Given $0<c<\epsilon$, we define the function $D^c:\mathcal{S}_0\setminus \mathcal{C}^c  \rightarrow [0,\infty]$ by 
 \[
 D^c(\pa ) := \inf \Bigl\{ \mbox{length}^{\tilde{g}^c}(\psi ) \, \vert \, \psi :[0,1] \rightarrow \mathcal{L}_{\pa} \mbox{ abs. cont.  } \, \psi(0)=\pa, \, \psi(1) \in Y^c \Bigr\},
 \]
 for every $\pa \in \mathcal{S}_0 \setminus \mathcal{C}^c$, and we denote its domain, the set of points $\pa \in \mathcal{S}_0\setminus C^c$, where $D^c(\pa)$ is finite, by $\mbox{dom}(D^c)$. Then, we call $\mathcal{L}_{\pa}$-geodesic a curve $\psi:[0,1] \rightarrow \mathcal{L}_{\pa}$ which is geodesic with respect to the metric $\tilde{g}^{c,\pa}$ induced by $\tilde{g}^c$ on $\mathcal{L}_{\pa}$, for any point $\pa \in Y^c$ we denote by $\exp_{\pa}^c: T_{\pa}\mathcal{L}_{\pa} \rightarrow \mathcal{L}_{\pa}$ the exponential map from $\pa$ with respect to $\tilde{g}^{c,\pa}$ and by considering $\vec{\mathcal{K}}_{Y^c}$ as a subbundle of $T\DeltaPerp$ (that is, for every $\pa \in Y^c$ we take $\vec{\mathcal{K}}(\pa)$) we define the smooth mapping $\mbox{Exp}^c: \vec{\mathcal{K}}_{Y^c} \rightarrow \mathcal{S}_0$ by
 \[
 \mbox{Exp}^c (\pa,\zeta) := \exp_{\pa}^c  (\zeta) \qquad \forall (\pa,\zeta) \in \vec{\mathcal{K}}_{Y^c}.
 \]
By completeness of $(\mathcal{S}_0,\tilde{g}^c)$, see Lemma \ref{lem:MetricReduction} (i), for every $\pa \in \mbox{dom}(D^c)$ any sequence $\{\psi_k:[0,1] \rightarrow \mathcal{L}_{\pa}\}_{k\in \N}$ of absolutely continuous curves such that 
 \[
 \psi_k(0)=\pa, \quad \psi_k(1) \in Y^c \quad \mbox{and} \quad \lim_{k\rightarrow \infty}  \mbox{length}^{g^c}(\psi_k ) =D^c(\pa),
 \]
 converges, up to taking a subsequence, to a $\mathcal{L}_{\pa}$-geodesic $\bar{\psi}$, called minimizing geodesic for $D^c(\pa)$, satisfying  $\bar{\psi}(0)=\pa$, $\bar{\psi}(1) \in \overline{Y^c}$ and $\mbox{length}^{g^c}(\psi_k ) =D^c(\pa)$. Moreover if in addition $D^c(\pa) < \bar{l}+2$ then we have $\bar{\psi}(1) \in Y^c$ because $\pa \in \mbox{dom}(D^c) \subset \mathcal{S}_0 \setminus \mathcal{C}^c$. For every $\pa \in \mbox{dom}(D^c)$, we set
\begin{equation}\label{eq:GammaI}
\begin{aligned}
\Gamma^c(\pa) & \quad \text{the set of all minimizing geodesics for} \,D^c(\pa),\\
\mathcal{I}^c(\pa) &:= \Bigl\{ \psi(t) \, \vert \, \psi \in \Gamma^c(\pa), \, t \in [0,1]\Bigr\}.
\end{aligned}
\end{equation} 
 By completeness of $(\mathcal{S}_0,\tilde{g}^c)$ and regularity of the foliation given by $\vec{\mathcal{K}}_{|\mathcal{S}_0}$, the mapping $\pa \in \mbox{dom}(D^c) \mapsto \mathcal{I}^c(\pa)$ has closed graph. Moreover, by the above construction and properties (iii)-(iv) of Lemma \ref{lem:MetricReduction}, the set 
 \[
 \mathcal{I}^c\left(\check{\mathcal{A}}^{\bar{\pa},c}\right) := \bigcup_{\pa \in \check{\mathcal{A}}^{\bar{\pa},c}}  \mathcal{I}^c(\pa)
 \]
 is a  compact subset of $\mathcal{S}_0$ which is contained in $\mbox{dom}(D^c)$. The following lemma follows from classical results on distance functions from submanifolds in Riemannian geometry (see Figure \ref{fig2}).
 
 \begin{lemma}\label{lem:MetricProperties}
 For every $0<c<\epsilon$, there are a relatively compact open subset $\mathcal{V}^{c,\bar{\pa}}$ of $\mathcal{T}_{\bar{\pa}}$  containing $ \check{\mathcal{A}}^{\bar{\pa},c}$, an open neighborhood $\mathcal{H}^c$ of $\bar{\pa}$ in $\mathcal{L}_{\bar{\pa}}$, an open set $\mathcal{U}^c\subset \mathcal{W}$  and a set $F^c \subset \mathcal{U}^c$ satisfying the following properties:
 \begin{itemize}
 \item[(i)] The set $F^c \subset \mathcal{U}^c$ is closed with respect to the induced topology on $\mathcal{U}^c$.
 \item[(ii)] The set $F^c$ has Lebesgue measure zero in $\mathcal{U}^c$.
 \item[(iii)] The function $D^c$ is smooth on the open set (recall the notation introduced for local transition maps \eqref{24mars5})
 \[
 \mathcal{U}^c \setminus F^c \quad \mbox{with} \quad \mathcal{U}^c := \bigcup_{\pa \in \mathcal{H}^c} T^{\bar{\pa},\pa}(\mathcal{V}^{c,\bar{\pa}}) = \bigcup_{\pa \in \mathcal{H}^c} \mathcal{V}^{c,\bar{\pa}}_{\pa} \subset \mathcal{W}
 \]
and for every $\pa \in \mathcal{U}^c \setminus F^c$ the set  $\Gamma^c(\pa)$ given in \eqref{eq:GammaI} is a singleton $\{\psi^{c,\pa}\}$, where  $\psi^{c,\pa}:[0,1] \rightarrow \mathcal{L}_{\pa}$ is the $\mathcal{L}_{\pa}$-geodesic (uniquely) defined by the initial conditions 
 \[
 \psi^{c,\pa}(0)=\pa \quad \mbox{and} \quad \dot{\psi}^{c,\pa}(0)=-\nabla D^c_{\pa}(\pa)
 \]
 ($\nabla D^c_{\pa}$ stands for the gradient of $D^c_{\pa}$ with respect to $g^{c,\pa}$).
\item[(iv)] For every $\pa \in \mathcal{H}^c$, the mapping 
\[
\begin{array}{rcl}
H^c \, : \,  \left( \left( \mathcal{U}^c\setminus F^c\right) \cap \mathcal{T}_{\pa}\right) \times [0,1] \, & \longrightarrow & \,  \mathcal{S}_0\\
(\pa',t)  \,  & \longmapsto  & \,  \psi^{c,\pa'} (t)  
\end{array}
\]
is a smooth diffeomorphism onto its image. 
 \end{itemize} 
 \end{lemma}
 
\begin{figure}[H]
\begin{center}
\includegraphics[width=10cm]{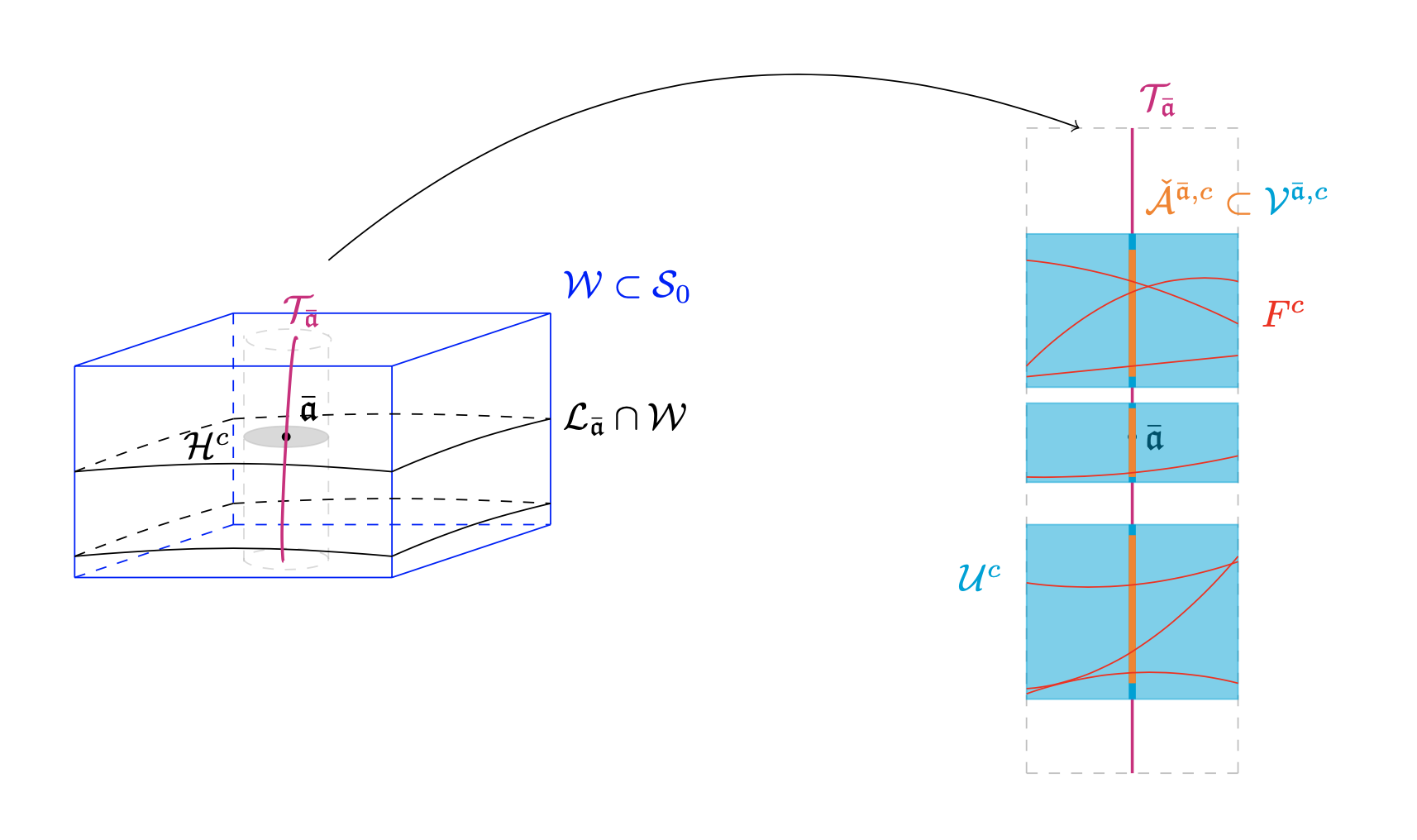}
\caption{A picture to illustrate Lemma \ref{lem:MetricProperties} \label{fig2}}
\end{center}
\end{figure}

 \begin{proof}[Proof of Lemma \ref{lem:MetricProperties}]
 The set $ \mathcal{I}^c\left(\check{\mathcal{A}}^{\bar{\pa},c}\right)  \cap \overline{Y^c}$ is a compact set which does not intersect $\partial Y^c$, so there is an open set $\mathcal{O}\subset \mathcal{S}_0$ which contains $\partial Y^c$ and such that $ \mathcal{I}^c\left(\check{\mathcal{A}}^{\bar{\pa},c}\right)  \cap Y^c \subset Y^c \setminus \overline{\mathcal{O}}$. As a consequence, by regularity of the mapping $\pa \in \mbox{dom}(D^c) \mapsto \mathcal{I}^c(\pa)$, there is an open set $\mathcal{U}^c\subset \mathcal{W}$ containing $\check{\mathcal{A}}^{\bar{\pa},c}$ such that
 \[
 D^c(\pa) < \bar{l}+2 \quad \mbox{and} \quad \mathcal{I}^c(\pa) \cap \overline{Y^c} \subset Y^c \setminus \overline{\mathcal{O}} \qquad \forall \pa \in \mathcal{U}^c.
 \]
In fact, for every $\pa \in \mathcal{U}^c$, the restriction of $D^c$ to the local leaf $\mathcal{L}_{\pa}\cap \mathcal{U}^c$, let us denote it by $D^c_{\pa}$,  coincides with the distance function to the set $\widetilde{Y}^{c,\pa}:=\mathcal{L}_{\pa} \cap (Y^c\setminus \overline{\mathcal{O}})$ which, by the transversality property given by Lemma \ref{lem:WitnessReduction} (iv) and compactness of $Y^c\setminus \mathcal{O} \subset Y^c$, is the union of finitely many points. So, as a distance function from a smooth submanifold (of dimension zero) on a complete Riemannian manifold, for every $\pa \in \mathcal{U}^c$ the function $D^c_{\pa}$ satisfies the following properties:
 \begin{itemize}
 \item[(P1)] The function $D^{c}_{\pa}$ is locally lipschitz on $\hat{\mathcal{L}}_{\pa} := \mathcal{L}_{\pa} \cap \mathcal{U}^c$ and its singular set $\Sigma(D^c_{\pa})$, defined as the set of points in $\hat{\mathcal{L}}_{\pa}$ where $D^{c}_{\pa}$ is not differentiable, has measure zero in $\hat{\mathcal{L}}_{\pa}$.
 \item[(P2)] Denoting by $\nabla D^c_{\pa}$ the gradient of $D^c_{\pa}$ with respect to $g^{c,\pa}$, define the limiting-gradient of $D^c_{\pa}$ at some point $\pa' \in \hat{\mathcal{L}}_{\pa}$, denoted by $\nabla^LD^c_{\pa}(\pa') \subset T_{\pa'}\mathcal{L}_{\pa}$, as the set of all limits in $T_{\pa'}\mathcal{L}_{\pa}$ of sequences of the form $\{\nabla D^c_{\pa}(\pa_k)\}_{k\in \N} \in T_{\pa_k}\mathcal{L}_{\pa}$ where $\{a_k\}_{k\in \N}$ is a sequence of points in $\hat{\mathcal{L}}_{\pa} \setminus  \Sigma(D^c_{\pa})$ converging to $\pa'$ (note that by (P1) such sequences do exist). Then a point $\pa' \in \hat{\mathcal{L}}_{\pa}$ belongs to $\Sigma(D^c_{\pa})$ if and only if $\nabla^LD^c_{\pa}(\pa')$ is not a singleton. Moreover, for every $\pa' \in \hat{\mathcal{L}}_{\pa}$, there is a one-to-one correspondence between $\nabla^LD^c_{\pa}(\pa')$ and $\Gamma^c(\pa')$ (the set of all minimizing geodesics for $D^c(\pa')$), namely a vector $\zeta' \in T_{\pa'}\hat{\mathcal{L}}_{\pa'}$ belongs to $\nabla^LD^c_{\pa}(\pa')$ if and only if the $\mathcal{L}_{\pa}$-geodesic $\psi^{\pa',\zeta'}:[0,1] \rightarrow \mathcal{L}_{\pa}$ (uniquely) defined by the initial conditions
 \[
 \psi^{\pa',\zeta'}(0)=\pa' \quad \mbox{and} \quad \dot{\psi}^{\pa',\zeta'}(0)=-\zeta'
 \]
 is a minimizing geodesic for $D^c(\pa')$. Moreover, every such geodesic satisfies 
 \[
 \psi^{\pa',\zeta'}(1) \in Z^{c,\pa}  \quad \mbox{and} \quad \psi^{\pa',\zeta'}(t) = \beta^{\pa',\zeta'}(1-t) \quad \forall t \in [0,1]
 \]
 where $\beta^{\pa',\zeta'}$ is the $\mathcal{L}_{\pa}$-geodesic given by 
 \[
 \beta(t) :=  \exp_{P(\pa',\zeta')}^c\left(V(\pa',\zeta')\right) \quad \forall t \in [0,1]
 \]
 with 
 \[
 P(\pa',\zeta') := \psi^{\pa',\zeta'}(1) \in Z^{c,\pa}  \quad \mbox{and} \quad V(\pa',\zeta') := -\dot{\psi}^{\pa',\zeta'}(1) \in T_{P(\pa',\zeta')} \mathcal{L}_{\pa}.
 \]
  \item[(P3)] Let $\mbox{Conj}^{c,\pa}(\hat{\mathcal{L}_{\pa}}) \subset \hat{\mathcal{L}_{\pa}}$ be the set of points $\pa'\in \hat{\mathcal{L}_{\pa}}$  for which there is $\zeta' \in \nabla^LD^c_{\pa}(\pa')$, called conjugate limiting-gradient of $D^c_{\pa}$ at $\pa'$, such that the tangent vector $V(\pa',\zeta')$ is a critical point of  the exponential map $\exp_{P(\pa',\zeta')}^c$. Then we have 
 \[
\mbox{Cut}^{c,\pa} \left(\hat{\mathcal{L}_{\pa}}\right) := \overline{ \Sigma(D^c_{\pa})} = \Sigma(D^c_{\pa}) \cup \mbox{Conj}^{c,\pa} \left(\hat{\mathcal{L}_{\pa}}\right).
 \] 
 \item[(P4)] The set $\mbox{Cut}^{c,\pa}(\hat{\mathcal{L}_{\pa}})$, called cut locus in $\hat{\mathcal{L}_{\pa}}$, has Lebesgue measure zero in $\hat{\mathcal{L}}_{\pa}$ and the function $D^c_{\pa}$ is smooth on $\hat{\mathcal{L}}_{\pa} \setminus \mbox{Cut}^{c,\pa}(\hat{\mathcal{L}_{\pa}})$. In particular, for every $\pa' \in \hat{\mathcal{L}}_{\pa}\setminus \mbox{Cut}^{c,\pa}(\hat{\mathcal{L}_{\pa}})$, the set $\nabla^LD^c_{\pa}(\pa')$ is a singleton $\{\zeta(\pa')\}$ and moreover the exponential map $\exp_{P(\pa',\zeta')}^c:T_{P(\pa',\zeta')}\mathcal{L}_{\pa} \rightarrow \mathcal{L}_{\pa}$ is a submersion at $V(\pa',\zeta')$.
 \end{itemize}
 
 The property (P1) follows from Rademacher's Theorem, (P2) may be found in \cite[Lemma 11]{rifford08} (where the result is stated with Hamiltonian viewpoint) and (P3)-(P4) may be found in \cite{cr10,mm03,sakai96}.\\
  
 To conclude the proof of the lemma, we define the set $F^c\subset \mathcal{U}^c$ by
 \[
 F^c := \bigcup _{\pa \in \mathcal{U}^c} \mbox{Cut}^{c,\pa}(\hat{\mathcal{L}_{\pa}}).
 \]
 Let us prove (i), that is, $F^c$ is closed in the topological subspace $\mathcal{U}^c$. Let $\{\pa_k\}_{k\in \N}$ be a sequence of points of $F^c$ converging to some $\pa \in \mathcal{U}^c$. Let us distinguish two cases:\\
 
\noindent Case 1:  There is a constant $\delta>0$ such that the diameters (with respect to $\tilde{g}^c$) of the sets $\nabla^LD^c_{\pa_k}(\pa_k)$ are all larger than $\delta$  (so that for all $k\in \N$, $\pa_k$ belongs to $\Sigma(D^c_{\pa_k})$).\\
Then $\pa$ admits two minimizing geodesics $\psi_1, \psi_2$ for $D^c(\pa)$  such that 
 \[
 \left|\dot{\psi}_1(0)-\dot{\psi}_1(0)\right|^{\tilde{g}^c}\geq \delta >0
 \]
 so $\nabla^LD^c_{\pa_k}(\pa)$ is not a singleton (by (P2)) and $\pa$ belongs to $\Sigma(D^c_{\pa})\subset F^c$. \\
 
 \noindent Case 2: There is not a constant $\delta>0$ such that the diameters (with respect to $\tilde{g}^c$) of the sets $\nabla^LD^c_{\pa_k}(\pa_k)$ are all larger than $\delta$  (so that for all $k\in \N$, $\pa_k$ belongs to $\Sigma(D^c_{\pa_k})$).\\
Then we have
 \[
\liminf_{k \rightarrow \infty} \mbox{diam}^{\tilde{g}^c} \nabla^LD^c_{\pa_k}(\pa_k) =0.
\] 
Let us again distinguish between two cases.\\

\noindent Subcase 2.1: There are infinitely many $k\in \N$ for which $\pa_k$ belongs to $\mbox{Conj}^{c,\pa_k}(\hat{\mathcal{L}}_{\pa_k})$.\\
Then, by considering a subsequence  of $\{V(\pa_k,\zeta_k)\}_{k\in \N}$ with $\zeta_k$ a conjugate limiting-gradient of $D^c_{\pa_k}$ at $\pa_k$, there is a tangent vector $V(\pa,\zeta)$ which is a critical point of  the exponential map $\exp_{P(\pa,\zeta)}^c$ as limit of the sequence of critical vectors $\{V(\pa_k,\zeta_k)\}_{k\in \N}$ (with respect to $\exp_{P(\pa_k,\zeta_k)}^c$). Therefore, $\pa$ belongs to $\mbox{Conj}^{c,\pa}(\hat{\mathcal{L}}_{\pa}) \subset F^c$ by (P3). \\

\noindent Subcase 2.2: The set of $k\in \N$ for which $\pa_k$ belongs to $\mbox{Conj}^{c,\pa_k}(\hat{\mathcal{L}}_{\pa_k})$ is finite.\\
If  $\pa \notin  \mbox{Cut}^{c,\pa}(\hat{\mathcal{L}_{\pa}})$, then by (P2) the limiting-gradient $\nabla^LD^c_{\pa}(\pa)$ is equal to a singleton $\{\zeta\}$ and there is only one minimizing geodesics for $D^c(\pa)$ given by $\psi^{\pa,\zeta}$. Thus, by (P3), up to considering a subsequence,  we may assume without loss of generality that for all $k\in \N$ there are $\zeta_k^1, \zeta_k^2$ in $\nabla^LD^c_{\pa_k}(\pa_k)$ such that 
\[
\zeta_k^1 \neq \zeta_k^2 \quad \mbox{and} \quad \lim_{k \rightarrow \infty}  \left|\zeta_k^1-\zeta_k^2\right|^{\tilde{g}^c} = 0
\]
and for $i=1,2$
\[
\lim_{k \rightarrow \infty} \zeta_i^k = \zeta, \quad \lim_{k \rightarrow \infty} P(\pa_k,\zeta_i^k)  =   P(\pa,\zeta), \quad \lim_{k \rightarrow \infty} V(\pa_k,\zeta_i^k) =V(\pa,\zeta).
\]
Since $\pa \notin  \mbox{Cut}^{c,\pa}(\hat{\mathcal{L}_{\pa}})$, (P4) shows that the exponential map $\exp_{P(\pa,\zeta)}^c:T_{P(\pa,\zeta)}\mathcal{L}_{\pa} \rightarrow \mathcal{L}_{\pa}$ is a submersion at $V(\pa,\zeta)$. So the mappings $\exp_{P(\pa_k,\zeta_k)}^c:T_{P(\pa_k,\zeta_k)}\mathcal{L}_{\pa_k} \rightarrow \mathcal{L}_{\pa_k}$ are submersions at $V(\pa_k,\zeta_k)$ for $k$ large enough but this is impossible because
\[
\exp_{P(\pa_k,\zeta_k)}^c \left(V(\pa_k,\zeta_1^k)\right)  = \exp_{P(\pa_k,\zeta_k)}^c \left(V(\pa_k,\zeta_2^k)\right)   \qquad \forall k \in \N.
\]
So we have  $\pa \in  \mbox{Cut}^{c,\pa}(\hat{\mathcal{L}_{\pa}})$.\\

To prove (ii), we just notice that $\mathcal{U}^c$ is foliated by the leaves $\hat{\mathcal{L}}_{\pa}$ whose intersection with $F^c$ has measure zero by (P4). So, we get the result by a Fubini argument. 

The point (iii) is a consequence of the fact that $\nabla^LD^c_{\pa}(\pa)$ is a singleton $\{\zeta^{\pa}\}$ for all $\pa$ in the open set $\mathcal{U}^c \setminus F^c$ together with the fact that $\exp_{P(\pa,\zeta^{\pa})}^c:T_{P(\pa,\zeta^{\pa})}\mathcal{L}_{\pa} \rightarrow \mathcal{L}_{\pa}$ is a submersion at $V(\pa,\zeta^{\pa})$ which implies that the mapping $\mbox{Exp}^c$ is a submersion at $(P(\pa,\zeta^{\pa}), V(\pa,\zeta^{\pa}))$. As a matter of fact, if $\pa \in \mathcal{U}^c \setminus F^c$ is fixed, then there is an open neighborhood $N$ of $(P(\pa,\zeta^{\pa}), V(\pa,\zeta^{\pa}))$ in $\vec{\mathcal{K}}_{Y^c} \subset T\DeltaPerp$ such that the image $\mbox{Exp}^c(N)$ is an open neighborhood of $\pa$ and we have necessarily for every $(P,V)\in N$,
\[
\nabla^LD^c_{A}(A) = \left\{\zeta^{A}\right\} = \left\{-\dot{\psi}^{(P,V)}(0)\right\} \quad \mbox{with} \quad A=A(P,V):=\mbox{Exp}^c (P,V),
\]
where $\psi^{(P,V)}:[0,1] \rightarrow \mathcal{L}_{(P,V)}$ is the $\mathcal{L}_{(P,V)}$-geodesic given by 
\[
\psi^{(P,V)}(t) := \mbox{Exp}^c (P,(1-t)V) = \exp_{(P,V)}^c  ((1-t)V) \qquad \forall t \in [0,1],
\]
because $\psi^{(P,V)}$ is the only $\mathcal{L}_{(P,V)}$-geodesic closed to $\psi^{\pa,\zeta^{\pa}}$ joining $A$ to $Z^{c,A}$. Since the mapping 
\[
A \, \longmapsto \, - \frac{d}{dt} \left\{ \psi^{\bigl(\mbox{Exp}^c\bigr)^{-1}(A)} \right\} (0)
\]
is smooth we infer that  $D^c$ is smooth on $\mathcal{U}^c\setminus F^c$ and that  $\Gamma^c(\pa)$  is a singleton for $\pa \in \mathcal{U}^c\setminus F^c$ because $\nabla^LD^c_{\pa}(\pa)$ is always a singleton (see (P1)). 

To prove (iv), we first notice that, up to shrinking $\mathcal{U}^c$, we may assume that for every $\pa \in \mathcal{H}^c$, the mapping $H^c$ is injective. As a matter of fact, suppose for contradiction that there are $\pa\in \mathcal{H}^c$ and 
\[
(\pa_1,t_1), (\pa_2,t_2) \in  \left( \left( \mathcal{U}^c\setminus F^c\right) \cap \mathcal{T}_{\pa}\right) \times [0,1]
\] 
such that
\[
H^c(\pa_1,t_1) =  \psi^{c,\pa_1} (t_1) =  \psi^{c,\pa2} (t_2)  =H^c(\pa_2,t_2).   
\]
Since $\psi^{c,\pa_1}$ and $\psi^{c,\pa_2}$ are minimizing the length (among curves with are horizontal with respect to the foliation) we have either $\pa_1=\pa_2$ and  $\psi^{c,\pa_1}  =  \psi^{c,\pa2} $ (because  $\Gamma^c(\pa_1)=\Gamma^c(\pa_2)$  is a singleton), or we have $\pa_1\neq \pa_2$ and $t_1=t_2=1$. In the latter case, we infer that $\pa_1$ and $\pa_2$ belong to the same leaf $\mathcal{L}_{\pa_1}=\mathcal{L}_{\pa_2}$ and can be connect by a curve horizontal (with respect to $\mathcal{L}_{\pa_1}$) of length $<2\bar{\ell}+4$. By Lemma \ref{lem:SecondReduction} (iii), this cannot occur if the open neighborhood $\mathcal{U}^c\subset \mathcal{W}$ of $\check{\mathcal{A}}^{\bar{\pa},c}$ is sufficiently small. The smoothness of $H^c$ follows from (iii) and the property of diffeomorphism is a consequence of the fact that all minimizing curves from $\mathcal{U}^c$ to $\bar{Y}^c\setminus \mathcal{O}$ have no conjugate times. 

We conclude easily the construction of $\mathcal{V}^{c,\bar{\pa}} \subset \mathcal{T}_{\bar{\pa}}$ and $\mathcal{H}^c$ of $\bar{\pa} \subset \mathcal{L}_{\bar{\pa}}$. 
 \end{proof}

The following lemma will allow us to conclude the proof of Theorem \ref{thm:Sardminrank}, it follows easily from Lemma \ref{lem:MetricProperties}.

\begin{lemma}\label{lem:MainReduction}
For every $0<c<\epsilon$, there are $\pa^c \in \mathcal{L}_{\bar{\pa}}\cap \mathcal{W}$, a finite set $J^c$, two collections of sets $\{\mathcal{O}_j^{c,0}\}_{j\in J^c}$, $\{\mathcal{O}_j^{c,1}\}_{j \in J^c}$ and a collection of functions $\{\Phi_j^c:\mathcal{O}_j^{c,0} \times [0,1] \rightarrow \mathcal{S}_0\}_{j \in J^c}$ satisfying the following properties: 
\begin{itemize}
\item[(i)]  The sets $\mathcal{O}_j^{c,0}$ (with $j\in J^c$) are pairwise disjoint.
\item[(ii)] The sets $\mathcal{O}_j^{c,1}$ (with $j\in J^c$) are pairwise disjoint.
\item[(iii)] For every $j\in J^c$, $\mathcal{O}_j^{c,0}$ is a compact, connected and oriented, smooth submanifold with boundary of $\mathcal{T}_{\pa^c}$ of dimension $r$.
\item[(iv)] For every $j\in J^c$, $\mathcal{O}_j^{c,1}$ is a compact, connected and oriented, smooth submanifold with boundary of $Y^c \subset X^c$ of dimension $r$.
\item[(v)] For every $j\in J^c$, $\Phi_j^c:\mathcal{O}_j^{c,0} \times [0,1] \rightarrow \mathcal{S}_0$ is smooth and for every $t\in [0,1]$, the restriction of $\Phi_j^c$ to $\mathcal{O}_j^{c,0} \times \{t\}$ is a diffeomorphism from $\mathcal{O}_j^{c,0} \times \{t\}$ to its image 
\[
\mathcal{O}_j^{c,t} := \Phi_j^c \left(\mathcal{O}_j^{c,0} \times \{t\}\right).
\]
Moreover, $\Phi_j^c(\pa,0) = \pa$ for every $\pa \in \mathcal{O}_j^{c,0}$ and  $\mathcal{O}_j^{c,1}$ is the diffeomorphic image of $\mathcal{O}_j^{c,0}\times \{1\}$ by $\Phi_j^c$.

\item[(vi)]  For every $j\in J^c$ and any $t,t' \in [0,1]$ with $t\neq t'$, $\mathcal{O}_j^{c,t} \cap \mathcal{O}_j^{c,t'}=\emptyset$.
\item[(vii)]  For every $j\in J^c$ and every $\pa \in \mathcal{O}_i^{c,0}$, the smooth curve $t\in [0,1] \rightarrow \Phi_i^c(\pa,t)$ is a $\mathcal{L}_{\pa}$-geodesic with non zero speed.
\item[(viii)] The set $\mathcal{O}^{c,0} := \cup_{j\in J^c}  \mathcal{O}_j^{c,0}$ has measure $\geq \nu/16$ with respect to the volume form $\eta_{\vert \mathcal{T}_{\bar{\pa}}}$.
\end{itemize}
\end{lemma}

\begin{proof}[Proof of Lemma \ref{lem:MainReduction}]
Fix $0<c<\epsilon$ and consider the sets $\mathcal{V}^{c,\bar{\pa}} \subset \mathcal{T}_{\bar{\pa}}$, $\mathcal{H}^c \subset \mathcal{L}_{\bar{\pa}}$, $F^c\subset \mathcal{S}_0$ and $\mathcal{U}^c \subset  \mathcal{W}$ given by Lemma \ref{lem:MetricProperties}. The set $\mathcal{U}^c$ is foliated by the leaves $\mathcal{V}^{c,\bar{\pa}}_{\pa}$ with $\pa \in  \mathcal{H}^c$ and by Lemma \ref{lem:MetricProperties} (ii), the set $F^c \cap \mathcal{U}^c$ has Lebesgue measure zero. Hence Fubini's Theorem implies that there is $\pa^c\in \mathcal{H}^c$ such that the set $F^c \cap \mathcal{V}^{c,\bar{\pa}}_{\pa^c} \subset \mathcal{T}_{\pa^c}$ has measure zero. Without loss of generality, up to shrinking $ \mathcal{V}^{c,\bar{\pa}}_{\pa^c}$ in $\mathcal{T}_{\pa^c}$ we may assume that the relatively compact open set $\mathcal{V}^{c,\bar{\pa}}_{\pa^c}$ satisfies
\[
\overline{\mathcal{V}^{c,\bar{\pa}}_{\pa^c}} \subset \mathcal{T}_{\pa^c},
\]
and moreover, by Lemma  \ref{lem:MetricReduction} (iii), we may assume by taking $\pa^c$ sufficiently close to $\bar{\pa}$  that the compact set 
\[
\check{\mathcal{A}}^{\bar{\pa},c}_{\pa^c}\subset \mathcal{V}^{c,\bar{\pa}}_{\pa^c} \subset \mathcal{T}_{\pa^c}
\]
has measure $\geq \nu/8$ with respect to the volume form $\eta_{\vert \mathcal{T}_{\pa^c}}$. Consider a smooth function $G_c:\mathcal{T}_{\pa^c} \rightarrow [0,\infty)$ such that
\[
G^{-1}_c \left(\{0\}\right) = \left(F^c\cap \mathcal{V}^{c,\bar{\pa}}_{\pa^c}\right) \cup \partial \mathcal{V}^{c,\bar{\pa}}_{\pa^c}
\]
and define for every $\epsilon>0$ the compact set (note that the set is compact because it does not intersect $\partial \mathcal{V}^{c,\bar{\pa}}_{\pa^c}$ where $G_c$ is vanishing)
\[
\Omega^{c}_{\epsilon} := G^{-1}_c \bigl( [\epsilon, \infty)\bigr) \cap \mathcal{V}^{c,\bar{\pa}}_{\pa^c}.
\]
By Sard's Theorem, $G_c$ admits a decreasing sequence $\{\epsilon_k\}_{k\in \N}$ of regular values converging to $0$. Thus we have 
\[
 \mathcal{V}^{c,\bar{\pa}}_{\pa^c} = \bigcup_{k\in \N} \Omega^{c}_{\epsilon_k} \quad \mbox{with} \quad \Omega^{c}_{\epsilon_k} \subset \Omega^{c}_{\epsilon_{k+1}} \quad \forall k \in \N 
\]
and for every $k\in \N$ the set $\Omega^{c}_{\epsilon_k}$ is a  compact, oriented, smooth submanifold with boundary of $\mathcal{T}_{\pa^c}$ of dimension $r$. As a consequence, since the measure of  $\mathcal{V}^{c,\bar{\pa}}_{\pa^c}$, which contains $ \check{\mathcal{A}}^{\bar{\pa},c}_{\pa^c}$, with respect to the volume form $\eta_{\vert \mathcal{T}_{\pa^c}}$ is $\geq \nu/8$, there is $\bar{k}\in \N$ large enough such that the measure (with respect to the volume form $\eta_{\vert \mathcal{T}_{\pa^c}}$) of the set
\[
\mathcal{O}^{c,0}  := \Omega^{c}_{\epsilon_{\bar{k}}}
\]
is $\geq \nu/16$. By construction, $\mathcal{O}^{c,0} $ is the union of finitely many components $\mathcal{O}^{c,0}_j$ satisfying properties (i), (iii), (viii) of the statement, where $j$ varies in a finite set $J^c$. Then, for every $j\in J^c$, we define $\Phi_j^c:\mathcal{O}_j^{c,0} \times [0,1] \rightarrow \mathcal{S}_0$ by
\[
\Phi_j^c := H^c_{| \mathcal{O}_j^{c,0} \times [0,1]}
\]
and we set
\[
\mathcal{O}_j^{c,1} :=  \Phi_j^c \left( \mathcal{O}_j^{c,0} \times \{1\} \right). 
\]
The properties (ii), (iv), (v), (vi) and (vii) are satisfied by the construction together with Lemma \ref{lem:MetricProperties} (iv).
\end{proof}

We are now ready to complete the proof of Theorem \ref{thm:Sardminrank}. Let us temporarily fix $0<c<\epsilon$. By Lemma \ref{lem:MainReduction}, there are a finite set $J^c$, two collections of sets $\{\mathcal{O}_j^{c,0}\}_{j\in J^c}$, $\{\mathcal{O}_j^{c,1}\}_{j \in^c}$ and a collection of functions $\{\Phi_j^c:\mathcal{O}_j^{c,0} \times [0,1] \rightarrow \mathcal{S}_0\}_{j\in J^c}$ such that the properties (i)-(viii) are satisfied. Set for every $j\in J^c$ (see Figure \ref{fig3})
\[
\mathcal{M}^c_j :=\Bigl\{ \Phi_j^c (\pa,t) \, \vert \, (\pa,t) \in \mathcal{O}_j^c \times [0,1]\Bigr\}. 
\]

\begin{figure}[H]
\begin{center}
\includegraphics[width=10cm]{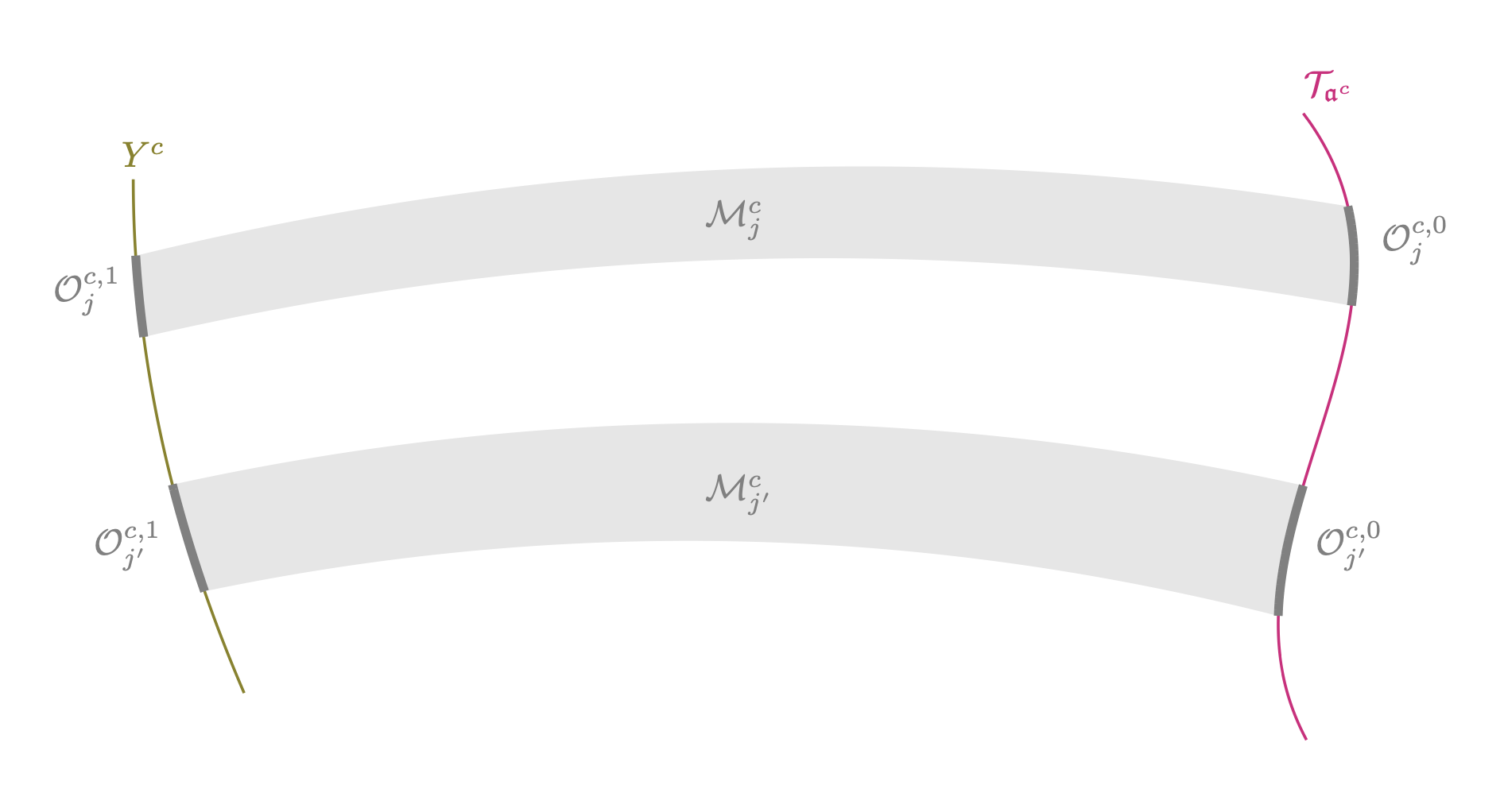}
\caption{The sets $\mathcal{O}_j^{c,0}, \mathcal{O}_j^{c,1}$ and $\mathcal{M}_j^{c}$ \label{fig3}}
\end{center}
\end{figure}

By properties (iii)-(vii), it is a topological manifold (with boundary) of dimension $r+1$ whose boundary can be written as 
\[
\partial \mathcal{M}_j^c = \mathcal{O}_j^{c,0} \cup \mathcal{O}_j^{c,1} \cup \mathcal{C}^c_j
\]
where both $\mathcal{O}_j^{c,0}$ and $\mathcal{O}_j^{c,1}$ are compact, connected, oriented, smooth submanifolds with boundary (Lemma \ref{lem:MainReduction} (iii)-(iv)) and where the cylindrical part $ \mathcal{C}^c_j$  given by
\[
 \mathcal{C}^c_j := \Bigl\{ \Phi_j^c (\pa,t) \, \vert \, (\pa,t) \in \partial \mathcal{O}_j^{c,0} \times (0,1)\Bigr\}
\]
is a smooth open oriented submanifold of dimension $r=2l$ satisfying
\[
\eta_{\vert   \mathcal{C}^c_j } = 0,
\]
because any point of $\mathcal{C}^c_j$ has the form $\Phi_j^c(\pa,t)$ with $\pa \in \partial \mathcal{O}_j^{c,0}$ and (by Lemma \ref{lem:MainReduction} (vii)) 
\[
0 \neq \frac{\partial  \Phi_j^c}{\partial t}(\pa,t) \in \left( T_{\Phi_j^c(\pa,t)}\mathcal{C}_j^c \right) \cap \left( T_{\Phi_j^c(\pa,t)}\mathcal{L}_{\pa}\right)
\]
\[
\mbox{with} \quad T_{\Phi_j^c(\pa,t)}\mathcal{L}_{\pa}=\vec{\mathcal{K}}\left(\Phi_j^c(\pa,t) \right) = \mbox{ker} \left(\omega^{\perp}_{\Phi_j^c(\pa,t)}\right).
\]
As a consequence, by applying Stokes' Theorem we have for every $j\in J^c$,
\[
\int_{\mathcal{O}_j^{c,0}} \eta = \int_{\mathcal{O}_j^{c,1}} \eta,
\]
which imply (because, by Lemma \ref{lem:MainReduction} (i)-(ii),  the sets $\mathcal{O}_j^{c,0}$ (resp. $\mathcal{O}_j^{c,1}$) are pairwise disjoint)
\begin{eqnarray}\label{24mars1}
\int_{\mathcal{O}^{c,0}} \eta = \int_{\cup_{j\in J^c} \mathcal{O}_j^{c,0}} \eta = \sum_{j\in J^c} \int_{\mathcal{O}_j^{c,0}} \eta = \sum_{j\in J^c} \int_{\mathcal{O}_j^{c,1}} \eta = \int_{\cup_{j\in J^c} \mathcal{O}_j^{c,1}} \eta =  \int_{\mathcal{O}^{c,1}} \eta.
\end{eqnarray}
But, on the one hand, by Lemma \ref{lem:MainReduction} (viii), we have 
\[
\int_{\mathcal{O}^{c,0}} \eta \geq \frac{\nu}{16},
\]
and, on the other hand Lemma \ref{lem:WitnessReduction} (ii) together with equation (\ref{23mars1}) yield ($\tilde{g}_{\vert X^c}$ denotes the metric induced by $\tilde{g}$ on $X^c$)
\[
\left| \int_{\mathcal{O}^{c,1}} \eta \right| \leq \int_{\mathcal{O}^{c,1}} |\eta| \leq \int_{\mathcal{O}^{c,1}} \delta(c) \, d\mbox{vol}^{\tilde{g}_{\vert X^c}} \leq \int_{X^c} \delta(c) \, d\mbox{vol}^{\tilde{g}_{\vert X^c}} \leq \delta(c) \, C,
\]
which tends to zero as $c$ tends to zero. Thus (\ref{24mars1}) cannot be satisfied for all $c>0$, this is a contradiction.

\bibliographystyle{plain}
\bibliography{Abnormal-Bibliography}

\begin{thebibliography}{10}

\bibitem{agrachev14}
Andrei~A. Agrachev.
\newblock Some open problems.
\newblock In {\em Geometric control theory and sub-{R}iemannian geometry},
  volume~5 of {\em Springer INdAM Ser.}, pages 1--13. Springer, Cham, 2014.

\bibitem{bfpr18}
A.~Belotto~da Silva, A.~Figalli, A.~Parusi\'nski, and L.~Rifford.
\newblock Strong {S}ard conjecture and regularity of singular minimizing
  geodesics for analytic sub-{R}iemannian structures in dimension 3.
\newblock {\em Invent. Math.}, 229(1):395--448, 2022.

\bibitem{bprfollowup}
A.~Belotto~da Silva, A.~Parusiński, and L.~Rifford.
\newblock Abnormal singular foliations and the sard conjecture for generic
  co-rank one distributions.
\newblock {\em arXiv 2310.20284}, 2023.

\bibitem{bprfirst}
A.~Belotto~da Silva, A.~Parusiński, and L.~Rifford.
\newblock Abnormal subanalytic distributions in sub-riemannian geometry.
\newblock {\em hal-04881557}, 2025.

\bibitem{br18}
Andr\'e{} Belotto~da Silva and Ludovic Rifford.
\newblock The {S}ard conjecture on {M}artinet surfaces.
\newblock {\em Duke Math. J.}, 167(8):1433--1471, 2018.

\bibitem{bv20}
Francesco Boarotto and Davide Vittone.
\newblock A dynamical approach to the {S}ard problem in {C}arnot groups.
\newblock {\em J. Differential Equations}, 269(6):4998--5033, 2020.

\bibitem{cr10}
Marco Castelpietra and Ludovic Rifford.
\newblock Regularity properties of the distance functions to conjugate and cut
  loci for viscosity solutions of {H}amilton-{J}acobi equations and
  applications in {R}iemannian geometry.
\newblock {\em ESAIM Control Optim. Calc. Var.}, 16(3):695--718, 2010.

\bibitem{chavel06}
Isaac Chavel.
\newblock {\em Riemannian geometry}, volume~98 of {\em Cambridge Studies in
  Advanced Mathematics}.
\newblock Cambridge University Press, Cambridge, second edition, 2006.
\newblock A modern introduction.

\bibitem{gg73}
M.~Golubitsky and V.~Guillemin.
\newblock {\em Stable mappings and their singularities}, volume Vol. 14 of {\em
  Graduate Texts in Mathematics}.
\newblock Springer-Verlag, New York-Heidelberg, 1973.

\bibitem{hardt82}
Robert~M. Hardt.
\newblock Some analytic bounds for subanalytic sets.
\newblock In {\em Differential geometric control theory ({H}oughton, {M}ich.,
  1982)}, volume~27 of {\em Progr. Math.}, pages 259--267. Birkh\"auser Boston,
  Boston, MA, 1983.

\bibitem{Hirsch74}
Mw~Hirsch.
\newblock A stable analytic foliation with only exceptional minimal sets.
\newblock {\em Lecture Notes in Mathematics}, 468:9--10, 1975.

\bibitem{lmopv16}
Enrico Le~Donne, Richard Montgomery, Alessandro Ottazzi, Pierre Pansu, and
  Davide Vittone.
\newblock Sard property for the endpoint map on some {C}arnot groups.
\newblock {\em Ann. Inst. H. Poincar\'e{} C Anal. Non Lin\'eaire},
  33(6):1639--1666, 2016.

\bibitem{lrtPreprint}
Antonio Lerario, Luca Rizzi, and Daniele Tiberio.
\newblock Quantitative approximate definable choices, 2024.

\bibitem{lr98}
J.-M. Lion and J.-P. Rolin.
\newblock Int\'egration des fonctions sous-analytiques et volumes des
  sous-ensembles sous-analytiques.
\newblock {\em Ann. Inst. Fourier (Grenoble)}, 48(3):755--767, 1998.

\bibitem{mm03}
Carlo Mantegazza and Andrea~Carlo Mennucci.
\newblock Hamilton-{J}acobi equations and distance functions on {R}iemannian
  manifolds.
\newblock {\em Appl. Math. Optim.}, 47(1):1--25, 2003.

\bibitem{montgomery02}
Richard Montgomery.
\newblock {\em A tour of subriemannian geometries, their geodesics and
  applications}, volume~91 of {\em Mathematical Surveys and Monographs}.
\newblock American Mathematical Society, Providence, RI, 2002.

\bibitem{ov19}
Alessandro Ottazzi and Davide Vittone.
\newblock On the codimension of the abnormal set in step two {C}arnot groups.
\newblock {\em ESAIM Control Optim. Calc. Var.}, 25:Paper No. 18, 17, 2019.

\bibitem{rifford08}
Ludovic Rifford.
\newblock On viscosity solutions of certain {H}amilton-{J}acobi equations:
  regularity results and generalized {S}ard's theorems.
\newblock {\em Comm. Partial Differential Equations}, 33(1-3):517--559, 2008.

\bibitem{riffordbourbaki}
Ludovic Rifford.
\newblock Singuli\`eres minimisantes en g\'eom\'etrie sous-{R}iemannienne.
\newblock Number 390, pages Exp. No. 1113, 277--301. 2017.
\newblock S\'eminaire Bourbaki. Vol. 2015/2016. Expos\'es 1104--1119.

\bibitem{sakai96}
Takashi Sakai.
\newblock {\em Riemannian geometry}, volume 149 of {\em Translations of
  Mathematical Monographs}.
\newblock American Mathematical Society, Providence, RI, 1996.
\newblock Translated from the 1992 Japanese original by the author.

\bibitem{speiss99}
Patrick Speissegger.
\newblock The {P}faffian closure of an o-minimal structure.
\newblock {\em J. Reine Angew. Math.}, 508:189--211, 1999.

\bibitem{dm96}
Lou van~den Dries and Chris Miller.
\newblock Geometric categories and o-minimal structures.
\newblock {\em Duke Math. J.}, 84(2):497--540, 1996.

\bibitem{zz95}
I.~Zelenko and M.~Zhitomirski\u~i.
\newblock Rigid paths of generic {$2$}-distributions on {$3$}-manifolds.
\newblock {\em Duke Math. J.}, 79(2):281--307, 1995.

\end{thebibliography}

\end{document}